\documentclass[20pt, oneside]{article}   	
\usepackage{geometry}                		
\usepackage{amsmath}
\usepackage{amsthm}	
\usepackage{amssymb}
\usepackage{tikz}
\usepackage[toc,page]{appendix}
\newtheorem{thm}{Theorem}[section]

\newtheorem{rem}[thm]{Remark}
\newtheorem{prop}[thm]{Proposition}
\newtheorem{note}[thm]{Note}
\newtheorem{lem}[thm]{Lemma}
\newtheorem{cor}[thm]{Corollary}
\newtheorem*{claim}{Claim}
\newtheorem*{lemma}{Lemma}

\newtheorem*{theorem}{Theorem}
\newtheorem*{thmm}{Main Theorem}
\newtheorem*{ques}{Question}

\newcommand{\F}{\mathbb{F}}
\newcommand{\fl}{\mathfrak{l}}
\newcommand{\fa}{\mathfrak{a}}
\newcommand{\fp}{\mathfrak{p}}
\newcommand{\Gal}{\mathrm{Gal}}

\newcommand{\GL}{\mathrm{GL}}

\title{Surjectivity of the adelic Galois Representation associated to a Drinfeld module of prime rank}
\author{Chien-Hua Chen}

\begin{document}

\maketitle

\section*{Abstract}
In this paper, let $\phi$ be the Drinfeld module over $\mathbb{F}_{q}(T)$ of prime rank $r$ defined by $$\phi_T=T+\tau^{r-1}+T^{q-1}\tau^r.$$ We prove that under certain condition on $\F_q$, the adelic Galois representation 
$${\rho}_{\phi}:{\rm{Gal}}(\mathbb{F}_q(T)^{{\rm{sep}}}/\mathbb{F}_q(T))\longrightarrow \varprojlim_{\mathfrak{a}}{\rm{Aut}}(\phi[\mathfrak{a}])\cong {\rm{GL_r}}(\widehat{A})$$
is surjective. 
\section{Introduction}

In \cite{Ser72}, Serre proved his famous ``Open Image Theorem" for elliptic curves over number field without complex multiplication. Restricted to elliptic curves over $\mathbb{Q}$, the theorem says

\begin{theorem}[\cite{Ser72}]
  If $E$ is an elliptic curve over $\mathbb{Q}$ without complex multiplication, then its associated adelic Galois representation
    $$\rho_E:{\rm{Gal}}(\bar{\mathbb{Q}}/\mathbb{Q})\rightarrow \varprojlim_{{m}}{\rm{Aut}}(E[m])\cong{\rm{GL}}_2(\hat{\mathbb{Z}})$$
    has open image in ${\rm{GL}}_2(\hat{\mathbb{Z}})$   
\end{theorem}

The following is then a natural question:
\begin{ques}\label{ellq1}
 Is it possible to have an elliptic curve $E$ over $\mathbb{Q}$ such that $\rho_E(\Gal(\bar{\mathbb{Q}}/\mathbb{Q}))={\rm{GL}}_2(\hat{\mathbb{Z}})$?
\end{ques}

Serre \cite{Ser72} showed that the answer to the above question is negative, the main reason being the existence of a non-trivial relation between the $2$-torsion $E[2](\bar{\mathbb{Q}})$ and some $m$-torsion $E[m](\bar{\mathbb{Q}})$. Although Serre proved the adelic surjectivity problem is negative for elliptic curves over $\mathbb{Q}$, Greicious \cite{Gre10} proved it is true for elliptic curves over a ``good enough'' number field. 
 Several mathematicians studied the generalization of adelic surjectivity problem for abelian varieties of dimension $\geqslant 2$ over a number field, cf. \cite{AnDo20}, \cite{Ha11}, \cite{Zy15}, and \cite{LSTX20}. Currently, it is known that there is a genus $3$ hyperelliptic curve over $\mathbb{Q}$ whose Jacobian variety satisfies the adelic surjectivity. Not much is known for adelic surjectivity of higher dimensional abelian varieties. 

Let $A=\F_q[T]$, $F=\F_q(T)$ and $\hat{A}=\varprojlim_{\mathfrak{a}} A/\mathfrak{a}$.  Pink and R\"utsche \cite{PR09}  proved the Drinfeld module analogue of Serre's open image theorem:
 \begin{theorem}[\cite{PR09}]
Let $\phi$ be a Drinfeld $A$-module over $F$ of rank $r$ without complex multiplication, then its associated adelic Galois representation
$$\rho_\phi:{\rm{Gal}}(F^{\rm{sep}}/F)\rightarrow \varprojlim_{\mathfrak{a}}{\rm{Aut}}(\phi[\mathfrak{a}])\cong{\rm{GL}}_r(\widehat{A})$$
    has open image in ${\rm{GL}}_r(\hat{A})$
\end{theorem}

Therefore, it is reasonable to study the adelic surjectivity problem for Drinfeld modules over $F$ of rank $r$. The answer for rank $r=1$ is positive, which follows from Hayes' work \cite{Hay74} on the function field analogue of class field theory. More precisely, the Carlitz module $C_T=T+\tau$ has surjective adelic Galois image. Moreover, the answer for rank $r=2$ is also positive, assuming that $q=p^e\geqslant5$ is an odd prime power. This was proved by  Zywina \cite{Zy11} for the rank-$2$ Drinfeld module $\phi_T=T+\tau-T^{q-1}\tau^2$. Remark that for the case $q=2$ and $r=2$, the author \cite{Ch21} proved there is no rank-$2$ Drinfeld $A$-module over $F$ with surjective adelic Galois representation due to a similar reason as Serre's arguments for elliptic curves over $\mathbb{Q}$. Therefore, some restrictions on $q$ is necessary.

In \cite{Ch20}, we have proved the adelic sujectivity of Galois representation associated to a Drinfeld module of rank $3$ defined by $\varphi_T=T+\tau^{2}+T^{q-1}\tau^{3}$. Thus for the case of Drinfeld modules of rank $r \geqslant 4$, one may expect the same result will hold for Drinfeld module $\phi$ over $F$ defined by $\phi_T=T+\tau^{r-1}+(-1)^{r-1}T^{q-1}\tau^{r}$. Our main result is the following theorem that shows the expectation is true for prime rank Drinfeld modules:

\begin{thmm}
Let $q=p^e$ be a prime power, $A=\mathbb{F}_q[T]$, and $F=\mathbb{F}_q(T)$. Assume $r\geqslant 3$ is a prime number and $q\equiv 1 \mod r$, there is a constant $c=c(r)\in \mathbb{N}$ depending only on $r$ such that for $p>c(r)$ the following statement is true ${\rm{:}}$

Let $\phi$ be a Drinfeld $A$-module over $F$ defined by $\phi_T=T+\tau^{r-1}+T^{q-1}\tau^r$. Then the adelic Galois representation 
$${\rho}_{\phi}:{\rm{Gal}}(\mathbb{F}_q(T)^{{\rm{sep}}}/\mathbb{F}_q(T))\longrightarrow \varprojlim_{\mathfrak{a}}{\rm{Aut}}(\phi[\mathfrak{a}])\cong {\rm{GL_r}}(\widehat{A})$$
 is surjective.
\end{thmm}

The general idea comes from the proof for $r=3$. However, when we try to adapt the strategy of the proof for $r=3$ to $r\geqslant 3$, several technical problems arise. To explain these problems we briefly recall the idea of the proof for $r=3$. The problems happened when we try to prove the mod $\mathfrak{l}$ Galois representations are irreducible and surjective for all prime ideals $\mathfrak{l}$ of $A$. 

To prove irreducibility in the rank-3 case, we aim for a contradiction by assuming the mod $\mathfrak{l}$ representation $\bar{\rho}_{\varphi,\mathfrak{l}}$ is reducible. Hence the characteristic polynomial $\bar{P}_{\varphi,\mathfrak{p}}(x)\in\mathbb{F}_\mathfrak{l}[x]$ of Frobenius elements $\bar{\rho}_{\varphi,\mathfrak{l}}({\rm{Frob}}_{\mathfrak{p}})$ must contain a linear factor for all prime $\mathfrak{p}\neq (T)\ {\rm{or}}\ \mathfrak{l}$. Since the degree is $3$, we can explicitly write down a linear factor of $\bar{P}_{\varphi,\mathfrak{p}}(x)$ when $\deg_T(\mathfrak{p})=1$. Furthermore, we can also determine the characteristic polynomials concretely. Hence a contradiction can be deduced by comparing coefficients of $\bar{P}_{\varphi,\mathfrak{p}}(x)$ and its factorization. In the case $r\geqslant 4$, the assumption of reducibility does not imply the characteristic polynomial of a Frobenius element would contain a linear factor.

Following the strategy for the rank-3 case, one might approach the problem of surjectivity by contradiction. We assume the image ${\rm{Im}}\bar{\rho}_{\varphi,\mathfrak{l}}(G_F)$ is a proper subgroup of ${\rm{GL}}_r(\mathbb{F}_{\mathfrak{l}})$. Thus ${\rm{Im}}\bar{\rho}_{\varphi,\mathfrak{l}}(G_F)$ must be contained in some maximal subgroup of ${\rm{GL}}_r(\mathbb{F}_{\mathfrak{l}})$. By our knowledge of ${\rm{Im}}\bar{\rho}_{\varphi,\mathfrak{l}}(I_T)$ (see Lemma \ref{lem22}), we can see that $|{\rm{Im}}\bar{\rho}_{\varphi,\mathfrak{l}}(G_F)|$ is divisible by certain power of $|A/\mathfrak{l}|$. Hence we can use this property to rule out possible maximal subgroups. In rank-$3$ case, we can rule out all possible maximal subgroups of ${\rm{GL}}_3(\mathbb{F}_{\mathfrak{l}})$. As the rank increases, we have some maximal subgroups that can not be ruled out by merely considering their sizes. The reason is that the growth rate of $|{\rm{Im}}\bar{\rho}_{\varphi,\mathfrak{l}}(I_T)|$ is much slower than the growth rate of $p$-power component of $|{\rm{GL}}_r(\mathbb{F}_{\mathfrak{l}})|$ when $r$ increases. 

Another harder problem is the classification of maximal subgroups in ${\rm{GL}}_r(\mathbb{F}_{\mathfrak{l}})$. From \cite{BHR13} Theorem 2.2.19, we can see that those maximal subgroups are divided into $9$ classes in Aschbacher's theorem. The first $8$ classes (geometric classes) have general description, but the ninth class (special class) doesn't have a known description for arbitrary $r$ so far. (This is also the reason why the authors could only describe all the maximal subgroups for low dimensional finite classical groups.)

Fortunately, if we restrict to the case where $r$ is a prime number, then we can combine Pink's work (see \cite{PR09}, section 3) on surjectivity of mod $\mathfrak{l}$ representations  
$$\bar{\rho}_{\phi,\mathfrak{l}}:G_F\longrightarrow {\rm{Aut}}(\phi[\mathfrak{l}])\cong {\rm{GL}}_r(\mathbb{F}_{\mathfrak{l}})$$
with Aschbacher's theorem (\cite{BHR13}, Theorem 2.1.5) to prove ``irreducibility of mod $\fl$ Galois representations'' toward ``surjectivity of mod $\fl$ Galois representations''  assuming the characteristic of $F$ is large enough. This procedure is described in section 3 and 4.

In section 5, we prove the irreducibility of the Galois representation. On the other hand, we prove the surjectivity of mod $(T)$ Galois representation directly using a result of Abhyankar. 

The proof toward $\mathfrak{l}$-adic surjectivity is similar to the rank-$3$ case and is proved in section 6. In section 7, we prove the adelic surjectivity under a further assumption $q\equiv 1 \mod r$. The proof for adelic surjectivity is similar to the rank-$3$ case as well.
 
\section{Preliminaries}
\subsection{Notation}

\begin{itemize}
\item$q=p^e$ is a prime power with $p\geqslant5$
\item$A=\mathbb{F}_q[T]$
\item$F=\mathbb{F}_q(T)$
\item$F^{{\rm{sep}}}$= separable closure of $F$
\item$F^{{\rm{alg}}}$= algebraic closure of $F$
\item$G_F= {\rm{Gal}}(F^{{\rm{sep}}}/F)$
\item $\fl=(l)$ a prime ideal of $A$, and define $\deg_T(\fl)=\deg_T(l)$ 
\item$A_{\mathfrak{p}}=$ completion of A at the nonzero prime ideal $\mathfrak{p}\triangleleft A$ 
\item$\widehat{A}=\underset{\mathfrak{a}\triangleleft A}{\varprojlim}A/\mathfrak{a}$
\item$F_{\mathfrak{p}}=$ fraction field of $A_{\mathfrak{p}}$
\item$\mathbb{F}_{\mathfrak{p}}= A/\mathfrak{p}$ 
\end{itemize}

\subsection{Drinfeld module over a field}
Let $K$ be a field, we call $K$ an {\textbf{$A$-field}} if $K$ is equipped with a homomorphism $\gamma: A \rightarrow K$. Let $K\{\tau\}$ be the ring of skew polynomials satisfying the commutation rule $c^q\cdot \tau=\tau \cdot c$.\\
A \textbf{Drinfeld $A$-module over $K$ of rank $r\geqslant 1$} is a ring homomorphism
\begin{align*}
\phi&: A\longrightarrow K\{\tau\} \\
      &\ \ \ a \mapsto \phi_a=\gamma(a)+ \sum^{r\cdot {\rm{deg}}(a)}_{i=1}g_i(a)\tau^i.
\end{align*}
It is uniquely determined by $\phi_T=\gamma(T)+\sum^{r}_{i=1}g_i(T)\tau^i$, where $g_r(T)\neq0$.\\ 
$\ker(\gamma)$ is called the {\textbf{$A$-characteristic}} of $K$, and we say $K$ has {\textbf{generic characteristic}} if $\ker(\gamma)=0$.\\

\begin{prop}\label{prop0.1}
Let $\phi$ be a Drinfeld module over $K$ of rank $r$ with nonzero $A$-characteristic $\mathfrak{p}$. For each $a\in A$, we may write $\phi$ as $\phi_{a}=c(a)\tau^{m(a)}+\cdots+C(a)\tau^{M(a)}$. There is a unique integer $0<h\leqslant r$ such that $m(a)=hv_{\mathfrak{p}}(a)$ for all nonzero $a\in A$, where $v_{\mathfrak{p}}$ is the $\mathfrak{p}$-adic valuation of $F$.
\end{prop}
This integer $h$ is called the {\textbf{height}} of $\phi$.

\begin{proof}
See \cite{Goss} Proposition 4.5.7.
\end{proof}

An \textbf{isogeny} from a Drinfeld module $\phi$ to another Drinfeld module $\psi$ over $K$ is an element $u\in K\{\tau\}$ such that $u\cdot \phi_a=\psi_a \cdot u \ \ \forall \ a\in A$. Hence the \textbf{endomorphism ring} of $\phi$ over $K$ is defined as
$${\rm{End}}_{K}(\phi)=\{u\in K\{\tau\}\ |\ u\cdot \phi_T=\psi_T \cdot u \}.$$
The Drinfeld module $\phi$ gives $K^{{\rm{alg}}}$ an $A$-module structure, where $a\in A$ acts on $K^{{\rm{alg}}}$ via $\phi_{a}$. We use the notation $^{\phi}K^{{\rm{alg}}}$ to emphasize the action of $A$ on $K^{{\rm{alg}}}$.\\
The {\textbf{$a$-torsion}} $\phi[a]=\{ {\rm{zeros}}\ {\rm{of}}\ \phi_a(x)= \gamma(a)x+ \sum^{r\cdot {\rm{deg}}(a)}_{i=1}g_i(a)x^{q^i} \}\subseteq K^{alg}$. The action $b\cdot \alpha=\phi_b(\alpha) \ \forall\ b\in A, \forall \alpha\in\phi[a]$ also gives $\phi[a]$ an $A$-module structure. \\

\begin{prop}\label{prop0.2}
Let $\phi$ be a rank $r$ Drinfeld module over $K$ and $\mathfrak{a}$ an ideal of $A$,
\begin{enumerate}
\item[1.] If $\phi$ has $A$-characteristic prime to $\mathfrak{a}$, then the $A/\mathfrak{a}$-module $\phi[\mathfrak{a}]$ is free of rank $r$
\item[2.] If $\phi$ has nonzero $A$-characteristic $\mathfrak{p}$, let $h$ be the height of $\phi$, then the $A/\mathfrak{p}$-module $\phi[\mathfrak{p}^e]$ is free of rank $r-h$ for all $e\in \mathbb{Z}_{\geqslant1}$. 
\end{enumerate}

\end{prop}
\begin{proof}
See \cite{Goss} Proposition 4.5.7.
\end{proof}

\begin{note}{\rm{From now on, we consider $K=F$ and $\gamma: A \rightarrow F$ is the natural injection map.}}
\end{note}

Let $\phi$ be a rank $r$ Drinfeld module over $F$ of generic characteristic, then $\phi_a(x)$ is separable, so we have $\phi[a]\subseteq F^{{\rm{sep}}}$.
This implies that $\phi[a]$ has a $G_F$-module structure. Given a nonzero prime ideal $\fl$ of $A$, we can consider the $G_F$-module $\phi[\mathfrak{l}]$. We obtain the so-called \textbf{mod $\mathfrak{l}$ Galois representation}
$$\bar{\rho}_{\phi,\mathfrak{l}}:G_F\longrightarrow {\rm{Aut}}(\phi[\mathfrak{l}])\cong GL_r(\mathbb{F}_{\mathfrak{l}}).$$
Taking inverse limit with respect to $\mathfrak{l}^i$, we have the \textbf{$\mathfrak{l}$-adic Galois representation}
$${\rho}_{\phi,\mathfrak{l}}: G_F \longrightarrow \varprojlim_{i}{\rm{Aut}}(\phi[\mathfrak{l^i}])\cong {\rm{GL_r}}(A_{\mathfrak{l}}).$$
Combining all representations together, we get the \textbf{adelic Galois representation}
$${\rho}_{\phi}:G_F\longrightarrow \varprojlim_{\mathfrak{a}}{\rm{Aut}}(\phi[\mathfrak{a}])\cong {\rm{GL_r}}(\widehat{A}).$$

\subsection{Carlitz module}
The \textbf{Carlitz module} is the Drinfeld module $C:A\longrightarrow F\{\tau\}$ of rank $1$ defined by $C_{T}=T+\tau$.
\begin{prop}\label{prop1}{{\rm{(Hayes \cite{Hay74})}}} 
For a nonzero ideal $\mathfrak{a}$ of $A$, the Galois representation
$$\bar{\rho}_{C,\mathfrak{a}}:G_F\longrightarrow {\rm{Aut}}(C[\mathfrak{a}])\cong (A/\mathfrak{a})^*$$ is surjective. 
This implies the adelic Galois representation of Carlitz module is surjective. Moreover, for prime ideals $\mathfrak{p}$ of $A$ such that $\mathfrak{p}\not|\;\mathfrak{a} $, we have $\bar{\rho}_{C,\mathfrak{a}}({{Frob}}_{\mathfrak{p}}) \equiv \mathfrak{p} \mod \mathfrak{a}$.
\end{prop}

\subsection{Reduction of Drinfeld modules}
Let $K$ be a local field with uniformizer $\pi$, valuation ring $R$, unique maximal ideal $\fp:=(\pi)$, normalized valuation $v$ and residue field $\F_\fp$. Let $\phi: A\longrightarrow K\{\tau\}$ be a Drinfeld module of rank $r$. We say that $\phi$ has \textbf{stable reduction} if there is a Drinfeld module $\phi':A\longrightarrow R\{\tau\}$ such that 
\begin{enumerate}
\item$\phi'$ is isomorphic to $\phi$ over $K$;
\item$\phi'$ mod $\fp$ is still a Drinfeld module (i.e. $\phi'_{T}$ mod $\mathfrak{p}$ has ${\rm{deg}}_{\tau}\geqslant 1$ ).
\end{enumerate}
$\phi$ is said to have \textbf{stable reduction of rank $r_1$} if $\phi$ has stable reduction and $\phi$ mod $\mathfrak{p}$ has rank $r_1$.\\
$\phi$ is said to have \textbf{good reduction} if $\phi$ has stable reduction and $\phi$ mod $\mathfrak{p}$ has rank $r$.

\begin{rem}\ \\
{\rm{We sometimes denote $\phi \mod \mathfrak{p}$ by $\phi \otimes \mathbb{F}_{\mathfrak{p}}$.}}

\end{rem}

The Drinfeld module analogue of N\`eron-Ogg-Shafarevich is the following:

\begin{prop}\label{prop3}{{\rm{(\cite{Tak82}, Theorem 1)}}} 
{{Let $\phi:A\longrightarrow K\{\tau\}$ be a Drinfeld module and $\mathfrak{l}$ be a nonzero prime ideal different from the $A$-characteristic of $\phi\otimes \F_\fp$. Then $\phi$ has good reduction if and only if the $\mathfrak{l}$-adic Galois representation is unramified at $\mathfrak{p}$}.
In other words, $\rho_{\phi,\mathfrak{l}}(I_{\mathfrak{p}})=1$  where $I_\mathfrak{p}$ is the inertia subgroup of $G_K$}.
\end{prop}

Let $u:\phi \rightarrow \psi$ be an isogeny between Drinfeld modules over $K$. We study the reduction type of the isogenous Drinfeld module. The isogeny $u$ induces an $G_K$-equivariant isomorphism between rational Tate modules
$$u:V_\mathfrak{l}(\phi)\rightarrow V_\mathfrak{l}(\psi),$$
where $V_\mathfrak{l}(\phi):=T_\mathfrak{l}(\phi)\otimes F_\fl$. Suppose $\mathfrak{l}$ is different from the $A$-characteristic of $\phi\otimes \F_\fp$, then the Drinfeld module analogue of N\`eron-Ogg-Shafarevich implies $\psi$ has good reduction at $\mathfrak{p}$. As a result, isogenous Drinfeld modules over a local field either both have good reduction or  bad reduction.

On the other hand, we prove that isogenous Drinfeld modules also preserve stable bad reduction under a condition on the inertia action. 

\begin{prop}\label{red}
Let $u:\phi \rightarrow \psi$ be an isogeny between Drinfeld modules over $K$. Suppose there is a prime ideal $\fl$ of $A$, and $\fl$ is different from the $A$-characteristic of $\phi\otimes\F_\fp$. Assume further that $\text{the inertia group $I_\fp$ acts on $V_{\fl}(\phi)$ via a group of unipotent matrices},$ then the isogenous Drinfeld module $\psi$ has stable reduction at $\fp$.

\end{prop}

\begin{proof}

We may assume $\psi$ is defined over the valuation ring $R$ after replacing $\psi$ by an isomorphic copy. Moreover, the assumption that $I_\fp$ acts on $V_\fl(\phi)$ by unipotent matrices implies that the action of $I_\fp$ on $V_\fl(\psi)$ is also unipotent. Therefore, $I_\fp$ acts by unipotent matrices on $\psi[\fl]$, which means the ramification index of $K(\psi[\fl])/K$ is a power on $p$.

Suppose $\psi$ does not have stable reduction over $K$, then any extension $L/K$ over which $\psi$ has stable reduction has ramification index divisible by some prime not equal to $p$. Therefore, if $\psi$ does not have stable reduction over $K$, then $\psi$ does not have stable reduction over $K(\psi[\fl])$ either, so we may assume that $\psi[\fl]$ is rational over $K$.

On the other hand, consider the Newton Polygon ${\rm NP}(\psi_\fl(x)/x)$ of the polynomial $\psi_{\fl}(x)/x$. If $\psi[\fl]$ is rational over K, then the roots of $\psi_\fl(x)$ have integer valuations, so the slope of the first line segment of ${\rm NP}(\psi_\fl(x)/x)$ is an integer. Otherwise, we have the slope of the first line segment of ${\rm NP}(\psi_\fl(x)/x)$ is a simplified fraction $\frac{c}{d}$ with denominator $d\neq 1$. Thus there are roots of $\psi_\fl(x)/x$ with valuation equal to $\frac{1}{d}$, which is a contradiction. 

Now we write $$\psi_\fl(x)=\sum_{i=0}^{r\cdot \deg_T(l)}g_i(l)x^{q^i}, \text{ with }g_0(l)=l.$$
The first line segment of ${\rm NP}(\psi_\fl(x)/x)$ has endpoints $(0,0)$ and $(q^m-1, v(g_m(l)))$. The integrality of slope implies $q^m-1 \mid v(g_m(l))$. Hence after taking a suitable isomorphic copy of $\psi$, we may assume $v(g_m(l))=0$. This implies one of the coefficient of $\psi_T(x)$ other than $Tx$ must be a unit. Hence we deduce that $\psi_T$ has stable reduction over $K$, which is a contradiction.

\end{proof}

\begin{rem}
Unlike abelian varieties (Corollaire 3.8 in chapter IX of \cite{SGA7}), stable bad reduction of Drinfeld module over local field does not imply the inertia group acts on Tate module via unipotent matrices.
The following is a counterexample:

Let $\fp=(T)$ and $\phi_T=T+\tau+T\tau^2$ be the Drinfeld $A$-module defined over $F_\fp$. It's clear that $\phi$ has stable bad reduction of rank $1$. For any prime $\fl=(T-c)\neq \fp$, the Tate uniformization shows that the inertia group $I_\fp$ acts on $T_\fl(\phi)$ via matrices of the form 
$$\left(\begin{array}{c|c}1 & * \\\hline 0 & c\end{array}\right).$$
We know $\det\circ\rho_{\phi,\fl}(I_\fp)=\rho_{\psi,\fl}(I_\fp)$, where $\psi$ is the rank-$1$ Drinfeld module $\psi_T=T-T\tau$ by Proposition 7.1 in \cite{Hei03}. Now we claim that $\rho_{\psi,\fl}(I_T)$ is nontrivial. We observe the Newton polygon of $\psi_{\fl}(x)/x$ with respect to the valuation $v_\fp$, the polygon is a single line with slope equal to $\frac{1}{q-1}$. Hence the Galois extension $F_\fp(\psi[\fl])/F_\fp$ ramifies. This implies the inertia group $I_T$ acts nontrivially on $\psi[\fl]$, hence acts nontrivially on $T_\fl(\psi)$ as well. Therefore, $\det\circ\rho_{\phi,\fl}(I_\fp)=\rho_{\psi,\fl}(I_\fp)$ is nontrivial, so $c\neq 1$.

\end{rem}

\subsection{Determinant of $\rho_{\phi}$}
Let $\phi$ be a Drinfeld module over $F$ defined by $$\phi_T=T+g_1\tau+g_2\tau^2+\cdots+g_r\tau^r.$$ 
Let $\mathfrak{p} \neq \mathfrak{l}$ be a prime of good reduction of $\phi$, the $\mathfrak{l}$-adic Galois representation ${\rho}_{\phi,\mathfrak{l}}$ is unramified at $\mathfrak{p}$ by Proposition \ref{prop3}. Therefore, the matrix ${\rho}_{\phi,\mathfrak{l}}({\rm{Frob}}_{\mathfrak{p}}) \in {\rm{GL}}_r(A_\mathfrak{l})$ is well-defined up to conjugation, so we can consider the characteristic polynomial $P_{\phi,\mathfrak{p}}(x)=\det(xI-{\rho}_{\phi,\mathfrak{l}}({\rm{Frob}}_{\mathfrak{p}}))$ of the Frobenius element ${\rm{Frob}_{\mathfrak{p}}}$.\\
The polynomial $P_{\phi,\mathfrak{p}}(x)$ has coefficients in $A$ which are independent of the choice of $\mathfrak{l}$. 
Moreover, $P_{\phi,\mathfrak{p}}(x)$ is equal to the characteristic polynomial of Frobenius endomorphism of $\phi\otimes\mathbb{F}_{\mathfrak{p}}$ acting on $T_{\mathfrak{l}}(\phi\otimes\mathbb{F}_{\mathfrak{p}})$. We may write $P_{\phi,\mathfrak{p}}(x)$ as follows:
$$P_{\phi,\mathfrak{p}}(x)=a_r+a_{r-1}x+a_{r-2}x^2+\cdots +a_{1}x^{r-1}+x^r\in A[x].$$
The constant term $a_r$ is equal to $(-1)^r\det\circ\bar{\rho}_{\phi,\mathfrak{a}}({\rm{Frob}}_{\mathfrak{p}}).$
\begin{prop}\label{prop3.1}{{\rm{(\cite{JKYu95}, Theorem 1)}}}
For $1\leqslant i\leqslant r$, we have $\deg (a_i)\leqslant\frac{i\cdot\deg (\mathfrak{p})}{r}$.

\end{prop}

\begin{prop}\label{prop4}{{\rm{(\cite{HsYu00}, p.268)}}} 
$$a_r=\epsilon(\phi)\cdot\mathfrak{p}.$$
Here $\epsilon(\phi)=(-1)^r(-1)^{{{\rm{deg}}_T\mathfrak{p}}(r+1)}{\rm{Nr}}_{\mathbb{F}_{\mathfrak{p}}/\mathbb{F}_q}(g_r)^{-1}$, which belongs to $\mathbb{F}_{q}^*$ and is independent of $\mathfrak{a}$.

\end{prop}

Let $\phi$ be the Drinfeld module over $F$ defined by $\phi_T=T+\tau^{r-1}+T^{q-1}\tau^r$, where $r$ is an odd prime. By \cite{Hei03} Proposition 7.1, we have 
$$\det\circ\bar{\rho}_{\phi,\mathfrak{a}}=\bar{\rho}_{\psi,\mathfrak{a}},$$
where $\psi$ is defined by $\psi_T=T+T^{q-1}\tau$. It's clear that $\psi$ is isomorphic over $F$ to the Carlitz module $C_T=T+\tau$, so we have the following commutative diagram:
$$
\begin{array}{ccc}
G_F &\xrightarrow{\bar{\rho}_{\phi,\mathfrak{l}}} & {\rm{GL}}_r(\mathbb{F}_{\mathfrak{l}}) \\
\|      &                         &            \downarrow\det         \\
G_F &\xrightarrow{\bar{\rho}_{C,\mathfrak{l}}} & (A/\mathfrak{l})^*
\end{array}.
$$
Hence we can deduce the following Corollary from Proposition \ref{prop1}.
\begin{cor}\label{cor5}
For prime ideals $\mathfrak{p}$ of $A$ such that $\mathfrak{p} \not|\ \mathfrak{a}$, we have $\det\circ\bar{\rho}_{\phi,\mathfrak{a}}({\rm{Frob}}_\mathfrak{p})\equiv \mathfrak{p} \mod \mathfrak{a}$.
\end{cor}

\subsection{Drinfeld-Tate uniformization}
Let $\phi: A\longrightarrow A_{\mathfrak{p}}\{\tau\}$ be a Drinfeld module. A \textbf{$\phi$-lattice} is a finitely generated free $A$-submodule of $^{\phi}F_{\mathfrak{p}}^{{\rm{sep}}}$ and stable under $G_{F_{\mathfrak{p}}}$-action. Here the discreteness is with respect to the topology of the local field $F_{\mathfrak{p}}^{{\rm{sep}}}$.\\
A \textbf{Tate datum} over $A_{\mathfrak{p}}$ is a pair $(\phi,\Gamma)$ where $\phi$ is a Drinfeld module over $A_{\mathfrak{p}}$ and $\Gamma$ is a $\phi$-lattice. Two pairs $(\phi,\Gamma)$ and $(\phi',\Gamma')$ of Tate datum are \textbf{isomorphic} if there is an isomophism from $\phi$ to $\phi'$ such that the induced $A$-module homomorphism $^{\phi}F_{\mathfrak{p}}^{{\rm{sep}}}\longrightarrow {^{\phi'}F_{\mathfrak{p}}^{{\rm{sep}}}}$ gives an $A$-module isomorphism $\Gamma \longrightarrow \Gamma'$.
\begin{prop}\label{prop6}{{\rm{(Drinfeld)}}}
{{Let $r_1$, $r_2$ be two positive integers. There is a one-to-one correspondence between two sets:}}

\begin{enumerate}
{\item The set of $F_{\mathfrak{p}}-$isomorphism classes of Drinfeld modules $\phi$ over $F_{\mathfrak{p}}$ of rank $r:=r_1+r_2$ with stable reduction of rank $r_1$
\item The set of $F_{\mathfrak{p}}-$isomorphism classes of Tate datum $(\psi,\Gamma)$ where $\psi$ is a Drinfeld module over $A_{\mathfrak{p}}$ of rank $r_1$ with good reduction, and $\Gamma$ is a $\psi$-lattice of rank $r_2$.
}
\end{enumerate}
\end{prop}
\begin{proof}
See chapter 4 in \cite{Leh09}.
\end{proof}
\begin{rem}\label{rem7}\ \\
{\rm{Fix $a\in A-\mathbb{F}_{q}$. From the proof of Proposition \ref{prop6}, we have the following two properties :
\begin{enumerate}
\item[(i)]
There is a $G_{F_{\mathfrak{p}}}$-equivariant short exact sequence of $A$-modules:
$$0\longrightarrow \psi[a]\longrightarrow \phi[a] \xrightarrow{\psi_a}\Gamma/a\Gamma \longrightarrow 0.$$
\item[(ii)]
There is an element $u\in  {A_{\mathfrak{p}}\{\{\tau \}\}}$ such that $u\psi_a=\phi_au$, here ${A_{\mathfrak{p}}\{\{\tau \}\}}$ is the set of power series in $\tau$ with coefficients in $A_{\mathfrak{p}}$.
This element $u$ induces an isomorphism of $A[G_{F_{\mathfrak{p}}}]$-modules from $\psi_a^{-1}(\Gamma)/\Gamma$ to $\phi[a]$ by mapping $z+\Gamma$ to $u(z)$.

Moreover, the function $u$ can be expressed in the following ways 

\begin{enumerate}
\item[1.] $u(x)=x+u_1x^q+u_2x^{q^2}+\cdots+u_ix^{q^i}+\cdots$, and $u_i$ belongs to the maximal ideal of $A_{\mathfrak{p}}$ for all $i$. 

\item[2.] When $r_2=1$, $u(x)=x\cdot\displaystyle  \prod_{0\neq a\in A}(1-\frac{x}{\phi_a(\gamma)})$ where $\gamma$ is a generator of $\Gamma$.

\end{enumerate}

\end{enumerate}

}}
\end{rem}

\section{Image of $\bar{\rho}_{\phi,\mathfrak{l}}$ and Aschbaher's Theorem}
From now on, we work under the following assumptions:

Let $r$ be a prime number, $A=\mathbb{F}_q[T]$, and $F=\mathbb{F}_q(T)$, where $q=p^e$ and $p>r!$. Let $\phi$ be a Drinfeld $A$-module over $F$ of rank $r$ with generic characteristic defined by $\phi_T=T+\tau^{r-1}+T^{q-1}\tau^r$. Let $\mathfrak{l}$ be a place of $F$ where $\phi$ has good reduction at $\mathfrak{l}$. We denote $\bar{\rho}_{\phi,\mathfrak{l}}(G_F)$ by $\Gamma_\mathfrak{l}$, so $\Gamma_\mathfrak{l}$ is a subgroup of ${\rm{GL}}_r(\mathbb{F}_{\mathfrak{l}})$

By Aschbacher's Theorem(Theorem \ref{asch} in Appendix B), $\Gamma_\mathfrak{l}$ lies in one of the Aschbacher classes. In this subsection, we'll show that classes $\mathcal{C}_2$, $\mathcal{C}_3$, $\mathcal{C}_4$, $\mathcal{C}_7$, and $\mathcal{C}_8$ can be ruled out.

\begin{lem}\label{lem22-1}
There is a basis of $\phi[\mathfrak{l}]$ such that 
$$\bar{\rho}_{\phi,\mathfrak{l}}(I_T) \subseteq \left\{\left(\begin{array}{cccc}1 &  &  & b_{1} \\ & \ddots &  & \vdots \\ &  & \ddots & b_{r-1} \\ &  &  & 1\end{array}\right),\ b_{i}\in \mathbb{F}_\mathfrak{l}\ \forall\ 1\leqslant\ i\leqslant\ r-1\right\}.$$ 
\end{lem}
\begin{proof}
As $\phi_T=T+\tau^{r-1}-T^{q-1}\tau^{r}$ has stable reduction at $(T)$ of rank $r-1$, we may use Tate uniformization to obtain a Tate datum $(\psi,\Gamma)$. Here $\psi$ has good reduction of rank $r-1$ and $\Gamma$ has rank $1$. The Drinfeld module $\psi:A\rightarrow A_{(T)}\{\tau\}$ has rank $r-1$ of good reduction, so the Galois representation $\bar{\rho}_{\psi,\mathfrak{l}}: G_{F_{(T)}} \rightarrow {\rm{Aut}}(\psi[\mathfrak{l}])$ is unramified. Thus there is a basis $\{w_1,w_2,\cdots,w_{r-1}\}$ of $\psi[\mathfrak{l}]$ such that $\sigma(w_i)=w_i\ \forall \sigma\in I_T,\ \forall\ 1\leqslant i\leqslant r-1$.\\
Now since $\Gamma$ is a free $A$-module of rank $1$, we may fix a generator $\gamma$ of $\Gamma$. Choose $z\in F_{(T)}^{\rm{sep}}$ such that $\psi_{\mathfrak{l}}(z)=\gamma$. The fact  that $\Gamma$ is stable under the Galois action implies that there is a character $\chi_\Gamma:G_{F_{(T)}} \rightarrow \mathbb{F}_q^*$ such that $\sigma(\gamma)=\chi_\Gamma(\sigma)\gamma,\ \forall\sigma\in I_T$.
By Remark \ref{rem7}(ii), we have
$$\psi_{\mathfrak{l}}(\sigma(z))=\sigma(\psi_{\mathfrak{l}}(z))=\sigma(\gamma)=\chi_\Gamma(\sigma)\gamma=\chi_\Gamma(\sigma)\psi_{\mathfrak{l}}(z)=\psi_\mathfrak{l}(\chi_\Gamma(\sigma)z) $$
Thus $\sigma(z)-\chi_\Gamma(\sigma)z\in \psi[\mathfrak{l}]$, therefore there are some elements $b_{\sigma,1}, b_{\sigma,2},\cdots, b_{\sigma,r-1}$ in $\mathbb{F}_{\mathfrak{l}}$ such that 
$$\sigma(z)=b_{\sigma,1}w_1+b_{\sigma,2}w_2+\cdots +b_{\sigma,r-1}w_{r-1}+\chi_{\Gamma}(\sigma)z.$$
Therefore, the action of $\sigma\in I_{T}$ on $\psi_\mathfrak{l}^{-1}(\Gamma)/\Gamma$ with respect to the basis $\{w_1+\Gamma, w_2+\Gamma,\cdots,w_{r-1}+\Gamma, z+\Gamma\}$ is of the form $\left(\begin{array}{cccc}1 &  &  & b_{\sigma,1} \\ & \ddots &  & \vdots \\ &  & 1 & b_{\sigma,r-1} \\ &  &  & \chi_\Gamma(\sigma)\end{array}\right)$ . 

Because of our choice of $\phi_T$, we have the determinant of mod $\mathfrak{l}$ Galois representation  $\det\circ\bar{\rho}_{\phi,\mathfrak{l}}$ is equal to the mod $\mathfrak{l}$ Galois representation of the Carlitz module $\bar{\rho}_{C,\mathfrak{l}}$. Hence $\chi_{\Gamma}(\sigma)=1$ for all  $\sigma\in I_{T}$ since the Carlitz module has good reduction at $(T)$. Therefore, we have deduced
$$\bar{\rho}_{\phi,\mathfrak{l}}(I_T) \subseteq \left\{\left(\begin{array}{cccc}1 &  &  & b_{1} \\ & \ddots &  & \vdots \\ &  & \ddots & b_{r-1} \\ &  &  & 1\end{array}\right),\ b_{i}\in \mathbb{F}_\mathfrak{l}\ \text{ for all }1\leqslant\ i\leqslant\ r-1\right\}$$ with respect to the basis $\{w_1+\Gamma, w_2+\Gamma,\cdots, w_{r-1}+\Gamma, z+\Gamma\}$.

\end{proof}

\begin{lem}\label{lem22}
The inclusion in Lemma \ref{lem22-1} is an equality. In particular, $\bar{\rho}_{\phi,\mathfrak{l}}(I_T)$ has order equal to $|\mathbb{F}_{\mathfrak{l}}|^{r-1}$. 
\end{lem}
\begin{proof}
So far we have from Lemma \ref{lem22-1} that 

$$\bar{\rho}_{\phi,\mathfrak{l}}(I_T) \subseteq \left\{\left(\begin{array}{cccc}1 &  &  & b_{1} \\ & \ddots &  & \vdots \\ &  & \ddots & b_{r-1} \\ &  &  & 1\end{array}\right),\ b_{i}\in \mathbb{F}_\mathfrak{l}\ \forall\ 1\leqslant\ i\leqslant\ r-1\right\}.$$ 

Let $F^{un}_{(T)}$ be the maximal unramified extension of $F_{(T)}$ in $F^{sep}_{(T)}$. Since $I_T= {\rm{Gal}}(F^{sep}_{(T)}/F^{un}_{(T)})$ by definition, we have $\bar{\rho}_{\phi,\mathfrak{l}}(I_T) \cong {\rm{Gal}}(F^{un}_{(T)}(\phi[\mathfrak{l}])/F^{un}_{(T)})$. By Remark \ref{rem7} and Lemma \ref{lem22-1}, we can derive that $F^{un}_{(T)}(\phi[\mathfrak{l}])=F^{un}_{(T)}(w_1,w_2,\cdots, w_{r-1},z)$. Furthermore, $w_i$ belongs to $F^{un}_{(T)}$ because $\bar{\rho}_{\psi,\mathfrak{l}}: G_{F_{(T)}} \rightarrow {\rm{Aut}}(\psi[\mathfrak{l}])$ is unramified for all $1\leqslant i\leqslant r-1$. Therefore, $F^{un}_{(T)}(\phi[\mathfrak{l}])=F^{un}_{(T)}(z)$ and its ramification index $e[F^{un}_{(T)}(z):F^{un}_{(T)}]$ at least the order of $v(z)$ in $\mathbb{Q}/\mathbb{Z}$ where $v$ is the normalized valuation of $F_{(T)}$. 

 From Remark \ref{rem7}(ii), we have $\phi_{\mathfrak{l}}(x)= \mathfrak{l}x\cdot{\displaystyle  \prod_{0\neq\gamma'\in \psi_{\mathfrak{l}}^{-1}(\Gamma)/\Gamma}}(1-\frac{x}{u(\gamma')})$. We compute the leading coefficients on both sides up to units of $A_{(T)}$:
\begin{equation}
T^{(q-1)( {\displaystyle  \sum_{i=1}^{\deg_T(\mathfrak{l})}q^{r(i-1)} })}=\pm\frac{\mathfrak{l}}{{\displaystyle  \prod_{0\neq\gamma'\in \psi_{\mathfrak{l}}^{-1}(\Gamma)/\Gamma}}u(\gamma')}. 
\end{equation}

As $\{w_1+\Gamma, w_2+\Gamma,\cdots, w_{r-1}+\Gamma, z+\Gamma\}$ is a basis of $\psi_{\mathfrak{l}}^{-1}(\Gamma)/\Gamma$, an element $\gamma'\in \psi_{\mathfrak{l}}^{-1}(\Gamma)/\Gamma$ can be written as a $\mathbb{F}_\mathfrak{l}$- linear combination of $w_i+\Gamma$ and $z+\Gamma$. Thus we have
$$
\begin{array}{ccl}
(q-1)( {\displaystyle  \sum_{i=1}^{\deg_T(\mathfrak{l})}q^{r(i-1)} })&=&v\left(\pm\frac{\mathfrak{l}}{{\displaystyle  \prod_{0\neq\gamma'\in \psi_{\mathfrak{l}}^{-1}(\Gamma)/\Gamma}}u(\gamma')}\right)\\
&=&0-\left( {{\displaystyle  \sum_{a_1,a_2,\cdots,a_{r-1},b\in \mathbb{F}_\mathfrak{l}\ {\rm{not\ all\ zero}}}}}v(u(a_1w_1+a_2w_2+\cdots+a_{r-1}w_{r-1}+bz))\right).
\end{array}
$$
Recall from the proof of Lemma \ref{lem22-1} that we have $\gamma=\psi_\mathfrak{l}(z)$, where $\gamma$ is a generator of the rank $1$ discrete $A$-module $\Gamma$. The discreteness of $\Gamma$ forces $v(\gamma)<0$, which implies $v(z)<0$ because every coefficient of $\psi_\mathfrak{l}(x)$ has nonnegative valuation. Moreover, the valuations $v(w_i)$ are all nonnegative because they are roots of $\psi_\mathfrak{l}(x)$.

The proof of $v(u(w_i))=v(w_i)$ is easy because $v(w_i)$ are non-negative and non-constant coefficients of $u(x)$ lie in the maximal ideal of $A_{(T)}$. 
For the proof of $v(u(z))=v(z)$, we use the chosen $\gamma$ to compare coefficients of two expressions of $u(x)$ in Remark \ref{rem7}(ii). We have 
$$v(u_n)\geqslant -(q^n-1)v(\gamma).$$
We also know that $v(\gamma)=q^{(r-1)(deg_T(\mathfrak{l}))}v(z)$. Hence
$$v(u_nz^{q^n})=v(u_n)+q^nv(z) \geqslant -(q^n-1)v(\gamma)+q^{(n-(r-1)(\deg_T(\mathfrak{l})))}v(\gamma).$$
For $n\geqslant1$, $v(u_nz^{q^n})$ is always non-negative. Therefore, $v(u(z))=v(z)$.

Thus we can compute the valuation of $u(a_1w_1+a_2w_2+\cdots+a_{r-1}w_{r-1}+bz)$ explicitly:
$$v(u(a_1w_1+a_2w_2+\cdots+a_{r-1}w_{r-1}+bz))=\begin{cases}
q^{(r-1)i}v(z), & {\rm{if}}\ b\neq0, \deg_T(b)=i\\
v(a_1w_1+a_2w_2+\cdots+a_{r-1}w_{r-1}), &{\rm{if}}\ b=0
\end{cases}
.$$
Hence we have
$$
\begin{array}{ccl}
(q-1)( {\displaystyle  \sum_{i=1}^{\deg_T(\mathfrak{l})}q^{r(i-1)} })&=&-(q^{(r-1)(\deg_T(\mathfrak{l}))})[(q-1)( {\displaystyle  \sum_{i=1}^{\deg_T(\mathfrak{l})}q^{r(i-1)} })]v(z)\\
&&-v\left({{\displaystyle  \prod_{a_1,a_2,\cdots,a_{r-1}\in \mathbb{F}_\mathfrak{l}\ {\rm{not\ all\ zero}}}}}(a_1w_1+a_2w_2+\cdots+a_{r-1}w_{r-1})\right).
\end{array}
$$
In fact, ${{\displaystyle  \prod_{a_1,a_2,\cdots,a_{r-1}\in \mathbb{F}_q\ {\rm{not\ all\ zero}}}}}(a_1w_1+a_2w_2+\cdots+a_{r-1}w_{r-1})$ is equal to the constant term $\mathfrak{l}$ of $\psi_{\mathfrak{l}}(x)/x$. Finally, we are able to compute the valuation $v(z)=-\frac{ 1 }{q^{(r-1)(\deg_T(\mathfrak{l}))}}$, its order in $\mathbb{Q}/\mathbb{Z}$ is equal to $q^{(r-1)(\deg_T(\mathfrak{l}))}=|A/\mathfrak{l}|^{r-1}$. Therefore, $|A/\mathfrak{l}|^{r-1} \leqslant e[F^{un}_{(T)}(z):F^{un}_{(T)}]$ and so $|\bar{\rho}_{\phi,\mathfrak{l}}(I_T)|\geqslant |A/\mathfrak{l}|^{r-1}$. Combining with Lemma \ref{lem22-1}, we have $|\bar{\rho}_{\phi,\mathfrak{l}}(I_T)|=|A/\mathfrak{l}|^{r-1}$.

 \end{proof}
 
 Now we can apply the Aschbacher's theorem (Theorem \ref{asch} in the Appendix) to rule out certain Aschbacher classes.

\begin{itemize}

\item $\Gamma_\mathfrak{l}$ does not lie in Class $\mathcal{C}_2$. 
\begin{proof}
Suppose $\Gamma_\mathfrak{l}$ lies in $\mathcal{C}_2$, then $\Gamma_\mathfrak{l}$ acting on $\mathbb{F}_\mathfrak{l}^r$ must be of the type ${\rm{GL}}_1(\mathbb{F}_\mathfrak{l})\wr S_r$, the wreath product of $\GL_1$ and the symmetric group $S_r$. Therefore, we have $|\Gamma_\mathfrak{l}|$ divides $|\mathbb{F}_\mathfrak{l}^*|\cdot {r!}$. We then get a contradiction from the fact that $p$ doesn't divide $|\mathbb{F}_\mathfrak{l}^*|\cdot {r!}$.
\end{proof}

\item $\Gamma_\mathfrak{l}$ does not lie in Class $\mathcal{C}_3$
\begin{proof}
Suppose $\Gamma_\mathfrak{l}$ lies in $\mathcal{C}_3$, then the action of $\Gamma_\mathfrak{l}$ on $\mathbb{F}_\mathfrak{l}^r$ must be of the type ${\rm{GL}}_1(\mathbb{F}_{\mathfrak{l}^r})$, here $\F_{\fl^r}$ is a degree-$r$ extension over $\F_\fl$. Thus we have $|\Gamma_\mathfrak{l}|$ divides ${\rm{GL}}_1(\mathbb{F}_{\mathfrak{l}^r})$, which contradicts to the fact that $p$ divides $|\Gamma_\mathfrak{l}|$.
\end{proof}

\item $\Gamma_\mathfrak{l}$ does not lie in Class $\mathcal{C}_4$

\begin{proof}
This is clear by the primality of $r$. Since $\mathbb{F}_\mathfrak{l}^r$ cannot have such tensor product decomposition 
$\mathbb{F}_\mathfrak{l}^r=V_1\otimes V_2$, where $V_1$(resp. $V_2$) is a $\mathbb{F}_\mathfrak{l}$-subspace of $\mathbb{F}_\mathfrak{l}^r$ of dimension $n_1$(resp. $n_2$) and $1<n_1<\sqrt{r}$.
\end{proof}

\item $\Gamma_\mathfrak{l}$ does not lie in Class $\mathcal{C}_7$

\begin{proof}
Suppose $\Gamma_\mathfrak{l}$ lies in $\mathcal{C}_2$, then the action of $\Gamma_\mathfrak{l}$ on $\mathbb{F}_\mathfrak{l}^r$ must be in a quotient of the standard wreath product ${\rm{GL}}_1(\mathbb{F}_\mathfrak{l})\wr S_r$. Hence we still have $|\Gamma_\mathfrak{l}|$ divides $|\mathbb{F}_\mathfrak{l}^*|\cdot {r!}$, a contradiction.

\end{proof}

\item $\Gamma_\mathfrak{l}$ does not lie in Class $\mathcal{C}_8$

\begin{proof}
Suppose $\Gamma_\mathfrak{l}$ lies in $\mathcal{C}_8$, then $\Gamma_\mathfrak{l}$ would preserve 
a non-degenerate classical form on $\mathbb{F}_\mathfrak{l}^r$ up to scalar multiplication. By classical form we mean symplectic form, unitary form or quadratic form. As $r$ is odd, $\mathbb{F}_\mathfrak{l}^r$ can only have unitary form or quadratic form structure. We refer to section 1.5 of \cite{BHR13} for the definitions and properties for classical forms on a vector space.

\begin{itemize}
\item[Case1.] $\Gamma_\mathfrak{l}$ preserves a non-degenerate unitary form on $\mathbb{F}_\mathfrak{l}^r$ up to scalar multiplication.

In this case, we are dealing with unitary form $<\cdot,\cdot>$ on $\mathbb{F}_\mathfrak{l}^r$. There is a basis $\mathcal{B}$ such that $<\cdot,\cdot>$ corresponds to the identity matrix $I_r$( see Proposition1.5.29 in \cite{BHR13}). Let $M\in \Gamma_\mathfrak{l}$ be a matrix with respect to the basis $\mathcal{B}$, the fact that $M$ preserves $<\cdot,\cdot>$ up to a scalar multiplication can be interpreted as the equality:

$$M\cdot M^{\sigma \top}=\lambda_M\cdot I.$$

Where ${\rm{id}} \neq\sigma \in {\rm{Aut}}(\mathbb{F}_\mathfrak{l})$ depends only on the unitary form $<\cdot,\cdot>$, $\sigma^2=1$ , and $\lambda_M \in \mathbb{F}_{\mathfrak{l}}^*$ depends on $M$. 

Therefore, we can compare the characteristic polynomials of such $M$ and $M^{-1}$. As $\lambda_M\cdot  M^{-1}=M^{\sigma \top}$, we have

$$\det(xI-M^{\sigma})=\frac{-x^r}{\det(\lambda_M^{-1}\cdot M)}\det(\frac{1}{x}I-\lambda_M^{-1}\cdot M)=\frac{-x^r}{\det(M)}\det(\frac{\lambda_M}{x}I-M)$$

Now we consider $M=\bar{\rho}_{\phi,\mathfrak{l}}({\rm{Frob}}_{\mathfrak{p}})$ where $\mathfrak{p}=(T-c)\neq (T)\ {\rm{or}}\ \mathfrak{l}$. The characteristic polynomial of $\bar{\rho}_{\phi,T}({\rm{Frob}}_{(T-c)})$ is congruent to $P_{\phi,(T-c)}(x)$ modulo $\fl$. By Proposition \ref{prop4}, we may write 
$$P_{\phi,(T-c)}(x)=-(T-c)+a_{r-1}x+a_{r-2}x^2+\cdots+a_1x^{r-1}+x^r\in A[x].$$ 
Proposition \ref{prop3.1} then implies all the $a_i$'s are belong to $\mathbb{F}_q$. Because $P_{\phi,(T-c)}(x)$ is also the characteristic polynomial of Frobenius endomorphism of $\phi\otimes\mathbb{F}_{\mathfrak{p}}$ acting on $T_{\fl}(\phi\otimes\mathbb{F}_{\mathfrak{p}})$, we have
$$-(\phi\otimes\mathbb{F}_{\mathfrak{p}})_{T-c}+(\phi\otimes\mathbb{F}_{\mathfrak{p}})_{a_{r-1}}\tau+(\phi\otimes\mathbb{F}_{\mathfrak{p}})_{a_{r-2}}\tau^2+\cdots+(\phi\otimes\mathbb{F}_{\mathfrak{p}})_{a_{1}}\tau^{r-1}+\tau^r=0.$$
As $\phi_T=T+\tau^{r-1}+T^{q-1}\tau^r$, we have $(\phi\otimes\mathbb{F}_{\mathfrak{p}})_{T-c}=\tau^{r-1}+\tau^r$. Thus $$a_{r-1}=a_{r-2}=\cdots=a_2=0 \text{ and } a_1=1.$$

Hence the characteristic polynomial of $\bar{\rho}_{\phi,\fl}({\rm{Frob}}_{(T-c)})\in \F_\fl[x]$ is 

$$\bar{P}_{\phi,(T-c)}(x)=x^r+x^{r-1}-\bar{\fp}, \text{ where $\bar{\fp}$ denotes the reduction of $\fp$ modulo $\fl$}.$$

Therefore, we have

$$x^r+x^{r-1}-\sigma(\bar{\mathfrak{p}})=\det(xI-M^{\sigma})=\frac{-x^r}{\det(M)}\det(\frac{\lambda_M}{x}I-M)= x^r-\frac{\lambda_M^{r-1}}{\bar{\mathfrak{p}}}x-\frac{\lambda_M^r}{\bar{\mathfrak{p}}}$$

This is a contradiction because the polynomials on both sides can not be equal.

\item[Case 2] $\Gamma_\mathfrak{l}$ preserves a non-degenerate quadratic form on $\mathbb{F}_\mathfrak{l}^r$ up to scalar multiplication

In this case, we are dealing with quadratic form $<\cdot,\cdot>$ on $\mathbb{F}_\mathfrak{l}^r$. 
There is a basis $\mathcal{B}$ such that $<\cdot,\cdot>$ corresponds to the the $r\times r$ matrix $$A:=\left(\begin{array}{cccc}1 & 0 & 0 & 0 \\0 & \ddots & 0 & 0 \\0 & 0 & 1 & 0 \\0 & 0 & 0 & d\end{array}\right),$$here $d$ is either $1$ or some nonsquare element in $\mathbb{F}_\mathfrak{l}^*$ . Let $M\in \Gamma_\mathfrak{l}$ be a matrix with respect to the basis $\mathcal{B}$, then $M$ preserves $<\cdot,\cdot>$ up to a scalar multiplication can be interpreted into the following equalities:

$$M\cdot A \cdot M^{\top}=\lambda_M\cdot A,$$
where $\lambda_M \in \mathbb{F}_{\mathfrak{l}}^*$ depends on $M$. Thus we have $M^\top=\lambda_M\cdot A^{-1}M^{-1}A$, which implies $M$ and $\lambda_M\cdot M^{-1}$ have the same characteristic polynomial.

$$\det(xI-M)=\frac{-x^r}{\det(\lambda_M^{-1}\cdot M)}\det(\frac{1}{x}I-\lambda_M^{-1}\cdot M)=\frac{-x^r}{\det(M)}\det(\frac{\lambda_M}{x}I-M).$$

Consider $M=\bar{\rho}_{\phi,\mathfrak{l}}({\rm{Frob}}_{\mathfrak{p}})$ where $\mathfrak{p}=(T-c)\neq (T)\ {\rm{or}}\ \mathfrak{l}$, we can deduce

$$x^r+x^{r-1}-\bar{\mathfrak{p}}=\det(xI-M)=\frac{-x^r}{\det(M)}\det(\frac{\lambda_M}{x}I-M)= x^r-\frac{\lambda_M^{r-1}}{\bar{\mathfrak{p}}}x-\frac{\lambda_M^r}{\bar{\mathfrak{p}}}.$$

Hence we have a contradiction.

\end{itemize}
\end{proof}
\end{itemize}

In summary, we've proved the following Proposition:
\begin{prop}\label{rk=prime}
Let $r$ be a prime number, $A=\mathbb{F}_q[T]$, and $F=\mathbb{F}_q(T)$, where $q=p^e$ and $p>r!$. Let $\phi$ be a Drinfeld $A$-module over $F$ of rank $r$ with generic characteristic defined by $\phi_T=T+\tau^{r-1}+T^{q-1}\tau^r$. Let $\mathfrak{l}$ be a place of $F$ where $\phi$ has good reduction at $\mathfrak{l}$. Denote $\bar{\rho}_{\phi,\mathfrak{l}}(G_F)$ by $\Gamma_\mathfrak{l}$.   Then $\Gamma_\mathfrak{l}$ can only lie in Aschbacher classes $\mathcal{C}_1$, $\mathcal{C}_5$,  $\mathcal{C}_6$, or  $\mathcal{S}$. 

\end{prop}

\section{Some algebraic group theory}
Under the assumption in Proposition \ref{rk=prime}, we know that $\Gamma_\mathfrak{l}$ can only lie in Aschbacher classes $\mathcal{C}_1$, $\mathcal{C}_5$,  $\mathcal{C}_6$, or  $\mathcal{S}$. Now we assume further that $\Gamma_{\mathfrak{l}}$ is irreducible, then $\Gamma_\mathfrak{l}$ can not lie in class $\mathcal{C}_1$. Hence $\Gamma_\mathfrak{l}$ lies in class $\mathcal{C}_5$,  $\mathcal{C}_6$, or  $\mathcal{S}$, which implies $\Gamma_\mathfrak{l}$ contains a subgroup that acts on $\mathbb{F}_\mathfrak{l}^r$ absolutely irreducibly. Hence $\Gamma_\mathfrak{l}$ acts on $\mathbb{F}_\mathfrak{l}^r$ absolutely irreducibly.

For most part, the following arguments are the same as in \cite{PR09}, p.888-p.889.

Consider the following exact sequence $0\rightarrow I_{\mathfrak{l}}^p\rightarrow I_\mathfrak{l}\rightarrow I_\mathfrak{l}^t\rightarrow 0$ of inertia group $I_\mathfrak{l}$. Here $I_\mathfrak{l}^p$ is the wild inertia group at $\mathfrak{l}$ and $I_\mathfrak{l}^t$ is the tame inertia group at $\mathfrak{l}$.
Let $h_\mathfrak{l}$ be the height of $\phi\otimes \mathbb{F}_\mathfrak{l}$, and let $\mathbb{F}_{n}$ be an extension of $\mathbb{F}_\mathfrak{l}$ in $\bar{\mathbb{F}}_\mathfrak{l}$ with $n:=|\mathbb{F}_\mathfrak{l}|^{h_\mathfrak{l}}$ elements. By \cite{PR09-2} Proposition 2.7, we have up to conjugation

$$\bar{\rho}_{\phi,\mathfrak{l}}(I_\mathfrak{l}^t)=\left(\begin{array}{cc}\mathbb{F}_n^* & 0 \\0 & 1\end{array}\right)\subseteq \Gamma_\mathfrak{l},$$
written in block matrices with size $h_\mathfrak{l}$, $r-h_\mathfrak{l}$. Since $\mathbb{F}_\mathfrak{l}^*\neq \{1\}$, the centralizer of $\bar{\rho}_{\phi,\mathfrak{l}}(I_\mathfrak{l}^t)$ in ${\rm{GL}}_{r,\mathbb{F}_\mathfrak{l}}$ is 
$$\left(\begin{array}{c|c}T_\mathfrak{l} & 0 \\\hline 0 & {\rm{GL}}_{r-h_\mathfrak{l},\mathbb{F}_\mathfrak{l}}\end{array}\right).$$

Here $T_\mathfrak{l}$ is the torus ${\rm{Res}}_{\mathbb{F}_\mathfrak{l}}^{\mathbb{F}_n} \mathbb{G}_{m,\mathbb{F}_n}$ in the algebraic group ${\rm{GL}}_{h_\mathfrak{l},\mathbb{F}_\mathfrak{l}}$. We embed $T_\mathfrak{l}$ into ${\rm{GL}}_{r,\mathbb{F}_\mathfrak{l}}$ by setting 
$$T_\mathfrak{l}=\left(\begin{array}{c|c}T_\mathfrak{l} & 0 \\\hline 0 & {\rm{I}}_{r-h_\mathfrak{l}}\end{array}\right)\subseteq {\rm{GL}}_{r,\mathbb{F}_\mathfrak{l}}.$$ 
The $\Gamma_\mathfrak{l}$-conjugacy class of $T_\mathfrak{l}$ in ${\rm{GL}}_{r,\mathbb{F}_\mathfrak{l}}$ is independent of the choice of place $\bar{\mathfrak{l}}$ of $\bar{F}$ above $\mathfrak{l}$.

Let $H_\mathfrak{l}^{\circ}$ be the connected algebraic subgroup of ${\rm{GL}}_{r,\mathbb{F}_\mathfrak{l}}$ generated by $\gamma T_\mathfrak{l} \gamma^{-1}$ for all $\gamma\in \Gamma_{\mathfrak{l}}$( see \cite{Hum81}, Proposition 7.5 ). Let $H_\mathfrak{l}$ be the algebraic subgroup of ${\rm{GL}}_{r,\mathbb{F}_\mathfrak{l}}$ generated by $H_\mathfrak{l}^{\circ}$ and $\Gamma_\mathfrak{l}$. Thus we have 
$$H_\mathfrak{l}^{\circ}\subseteq H_\mathfrak{l} \subseteq {\rm{GL}}_{r,\mathbb{F}_\mathfrak{l}}.$$
Since $H_\mathfrak{l}$ contains $\Gamma_\mathfrak{l}$, $H_\mathfrak{l}$ acts on $\mathbb{F}_\mathfrak{l}^r$ absolutely irreducibly. Now we can make the following claim

\begin{itemize}
\item[Claim:] $H_\mathfrak{l}^{\circ}$ acts on $\mathbb{F}_\mathfrak{l}^r$ absolutely irreducibly. In other words, $H_{\mathfrak{l},\bar{\mathbb{F}}_\mathfrak{l}}^{\circ}$ acts on $\bar{\mathbb{F}}_\mathfrak{l}^r$ irreducibly.

\begin{proof}[Proof of claim]
Let $W$ be a non-trivial $H_{\mathfrak{l},\bar{\mathbb{F}}_\mathfrak{l}}^{\circ}$-invariant subspace of $\bar{\mathbb{F}}_\mathfrak{l}^r$ of minimal dimension. As $H_{\mathfrak{l},\bar{\mathbb{F}}_\mathfrak{l}}^{\circ}$ is normalized by $\Gamma_\mathfrak{l}$, we know that $\gamma W$ is also invariant under $H_{\mathfrak{l},\bar{\mathbb{F}}_\mathfrak{l}}^{\circ}$ for all $\gamma\in \Gamma_\mathfrak{l}$. Consider $V=\sum_{\gamma\in \Gamma_\mathfrak{l}} \gamma W$, it is invariant under $\Gamma_\mathfrak{l}$. Thus we have $\bar{\mathbb{F}}_\mathfrak{l}^r = \sum_{\gamma\in \Gamma_\mathfrak{l}} \gamma W$ because $\Gamma_\mathfrak{l}$ acts on $\bar{\mathbb{F}}_\mathfrak{l}^r$ irreducibly. Since each $\gamma W$ is irreducible under the action of $H_{\mathfrak{l},\bar{\mathbb{F}}_\mathfrak{l}}^{\circ}$, there is a natural number $s_\mathfrak{l}$ and a decomposition 
$${\bar{\mathbb{F}}_\mathfrak{l}}^r=W_1\oplus W_2\oplus \cdots \oplus W_{s_\mathfrak{l}}$$

into irreducible $H_{\mathfrak{l},\bar{\mathbb{F}}_\mathfrak{l}}^{\circ}$-subspaces which are conjugate under $H_{\mathfrak{l},\bar{\mathbb{F}}_\mathfrak{l}}.$

Let $t_\mathfrak{l}$ be the common dimension of $W_i$. Then $H_{\mathfrak{l},\bar{\mathbb{F}}_\mathfrak{l}}^{\circ}$ acts on $\bar{\mathbb{F}}_\mathfrak{l}^r$ through matrices in ${\rm{GL}}_{t_\mathfrak{l},\bar{\mathbb{F}}_\mathfrak{l}}^{s_\mathfrak{l}}$. Thus $H_{\mathfrak{l},\bar{\mathbb{F}}_\mathfrak{l}}^{\circ}\subseteq {\rm{GL}}_{t_\mathfrak{l},\bar{\mathbb{F}}_\mathfrak{l}}^{s_\mathfrak{l}}$. The algebraic subgroup of ${\rm{GL}}_{r,\bar{\mathbb{F}}_\mathfrak{l}}$ mapping a summand $W_i$ to some summand $W_j$ is isomorphic to ${\rm{GL}}_{t_\mathfrak{l},\bar{\mathbb{F}}_\mathfrak{l}}^{s_\mathfrak{l}}\rtimes S_{s_\mathfrak{l}}$. 

\begin{lem}\label{semidirect}
$H_{\mathfrak{l},\bar{\mathbb{F}}_\mathfrak{l}}\subseteq {\rm{GL}}_{t_\mathfrak{l},\bar{\mathbb{F}}_\mathfrak{l}}^{s_\mathfrak{l}}\rtimes S_{s_\mathfrak{l}}$
\end{lem}
\begin{proof}
See Lemma 3.5 in \cite{PR09}.
\end{proof}

Combining the result $\Gamma_\mathfrak{l} \subseteq H_{\mathfrak{l},\bar{\mathbb{F}}_\mathfrak{l}}\subseteq {\rm{GL}}_{t_\mathfrak{l},\bar{\mathbb{F}}_\mathfrak{l}}^{s_\mathfrak{l}}\rtimes S_{s_\mathfrak{l}}$ with the primality of $r$, we can conclude either $s_\mathfrak{l}=1,\ t_\mathfrak{l}=r$ or $s_\mathfrak{l}=r,\ t_\mathfrak{l}=1$. If $s_\mathfrak{l}=r,\ t_\mathfrak{l}=1$, then we have
$$\Gamma_\mathfrak{l} \subseteq {\rm{GL}}_{1,\bar{\mathbb{F}}_\mathfrak{l}}^r\rtimes S_r.$$
Since we assume $p>r!$, there is no element in ${\rm{GL}}_{1,\bar{\mathbb{F}}_\mathfrak{l}}^r\rtimes S_r$ whose order is divisible by $p$. Thus we've deduced a contradiction because $\Gamma_\mathfrak{l}$ has elements of order divisible by $p$. Therefore we must have $s_\mathfrak{l}=1,\ t_\mathfrak{l}=r$. So $H_{\mathfrak{l},\bar{\mathbb{F}}_\mathfrak{l}}^{\circ}$ acts on ${\bar{\mathbb{F}}}_\mathfrak{l}^r$  irreducibly.
\end{proof}
\end{itemize}
\ \\

So far we have the following information on $H_\mathfrak{l}^{\circ}$:
\begin{itemize}
\item $H_\mathfrak{l}^{\circ}$ is a connected algebraic group over $\mathbb{F}_\mathfrak{l}$

\item $H_\mathfrak{l}^{\circ}$ acts on ${\mathbb{F}}_\mathfrak{l}^r$ absolutely irreducibly.

\item $H_{\mathfrak{l},\bar{\mathbb{F}}_\mathfrak{l}}^{\circ}$ has a cocharacter of weight $1$ with multiplicity $1$ and weight $0$ with multiplicity $r-1$(In the proof of Lemma \ref{semidirect})
\end{itemize}

With these conditions in hand, we can apply Proposition A.3 in \cite{P97} and prove that
$$H_\mathfrak{l}^{\circ}={\rm{GL}}_{r,\mathbb{F}_\mathfrak{l}}$$

\subsection{Image of residual representation }

\begin{thm}\label{surjmodl}
Let $q=p^e$ be a prime power, $A=\mathbb{F}_q[T]$, and $F=\mathbb{F}_q(T)$. Assume $r\geqslant 3$ is a prime number, there is a constant $c=c(r)\in \mathbb{N}$ depending only on $r$ such that for $p>c(r)$ the following statement is true ${\rm{:}}$

Let $\phi$ be a Drinfeld $A$-module over $F$ of rank $r$ defined by $\phi_T=T+\tau^{r-1}+T^{q-1}\tau^r$. Let $\mathfrak{l}$ be a finite place of $F$ where $\phi$ has good reduction at $\mathfrak{l}$. Then the mod $\mathfrak{l}$ representation 
$$\bar{\rho}_{\phi,\mathfrak{l}}:G_F\longrightarrow {\rm{Aut}}(\phi[\mathfrak{l}])\cong {\rm{GL}}_r(\mathbb{F}_{\mathfrak{l}})$$
 is either reducible or surjective.
\end{thm}

\begin{proof}

 Let $r$ be a prime number, $A=\mathbb{F}_q[T]$, and $F=\mathbb{F}_q(T)$ where $q=p^e$ where $p>r!$. Let $\phi$ be a Drinfeld $A$-module over $F$ of rank $r$ with generic characteristic defined by $\phi_T=T+\tau^{r-1}+T^{q-1}\tau^r$. Let $\mathfrak{l}$ be a place of $F$ where $\phi$ has good reduction at $\mathfrak{l}$. Let $\Gamma_\mathfrak{l}=\bar{\rho}_{\phi,\mathfrak{l}}(G_F)$ be irreducible. From the argument above, we have $H_\mathfrak{l}^{\circ}={\rm{GL}}_{r,\mathbb{F}_\mathfrak{l}}$. 
 
 Now we can apply Lemma 3.12 in \cite{PR09} for all places $\mathfrak{l}$ of $F$ where $\phi$ has good reduction. Thus there exists a natural number $\tilde{c}$ depends only on $r$ such that 
 $$[{\rm{GL}}_r(\mathbb{F}_\mathfrak{l}): \Gamma_\mathfrak{l}]\leqslant \tilde{c}.$$
Therefore, we replace $p$ by a prime number $p>c:={\rm{max}}\{r!, \tilde{c}!\}$ and run Lemma 3.12 in \cite{PR09} again, we get  $$[{\rm{GL}}_r(\mathbb{F}_\mathfrak{l}): \Gamma_\mathfrak{l}]\leqslant \tilde{c} \text{ and } |{\rm PGL}_r(\F_\fl)|>\tilde{c}!.$$
By applying Proposition 2.3 in \cite{PR09}, we have 
 $$\Gamma_\mathfrak{l}\supseteq {\rm{SL}}_r(\mathbb{F}_\mathfrak{l}).$$
By Proposition \ref{prop4} and Chebotarev density theorem, $\det\circ\bar{\rho_{\phi,\mathfrak{l}}}$ is equal to the mod $\mathfrak{l}$ representation $\bar{\rho_{C,\mathfrak{l}}}$ of Carlitz module. Hence we have $\det(\Gamma_{\mathfrak{l}})= \mathbb{F}_\mathfrak{l}^*$, which implies
$$\Gamma_\mathfrak{l}={\rm{GL}}_r(\mathbb{F}_\mathfrak{l}).$$
\end{proof}

\section{Irreducibility of the mod $\mathfrak{l}$ representations }
We assume $\mathfrak{l}\neq(T)$ from here to subsection 5.2, the case when $\mathfrak{l}=(T)$ will be dealt with in subsection 5.3.

For $\fl=(T-c)$ a degree-$1$ prime ideal different from $(T)$, the proof for irreducibility is straightforward.  

\begin{claim}
$\phi_{T-c}(x)/x=(T-c)+x^{q^{r-1}-1}+T^{q-1}x^{q^r-1}$ is irreducible over $F$.

\end{claim}

\begin{proof}[Proof of claim]
First of all, $\phi_{T-c}(x)/x$ is irreducible over $F$ if and only if its reciprocal polynomial 
$$x^{q^r-1}\phi_{T-c}(x^{-1})/x^{-1}=(T-c)x^{q^r-1}+x^{q^r-q^{r-1}}+T^{q-1}$$
is irreducible over $F$. As $T-c\in F^*$, it is enough for us to prove the polynomial 
$$x^{q^r-1}+\frac{1}{T-c}x^{q^r-q^{r-1}}+\frac{T^{q-1}}{T-c}$$
is irreducible over $F$.  As $F=\F_q(T)=\F_q(T-c)$, we may consider the completion $F'$ of $F$ at the place $(\frac{1}{T-c})$. Then the polynomial $x^{q^r-1}+\frac{1}{T-c}x^{q^r-q^{r-1}}+\frac{T^{q-1}}{T-c}$ is irreducible over $F'$ by the Eisenstein criterion. Hence $x^{q^r-1}+\frac{1}{T-c}x^{q^r-q^{r-1}}+\frac{T^{q-1}}{T-c}$ must be irreducible over $F$, the proof is now complete.

\end{proof}

For general prime ideal $\fl\neq(T)$, we suppose $\phi[\mathfrak{l}]$ viewed as a $\mathbb{F}_{\mathfrak{l}}[G_F]$ is reducible, then it contains a $G_F$-invariant proper submodule $X$. Let ${\rm{dim}}_{\F_\fl}X=d$, we claim that $d\neq 1 \text{ or } r-1$. This follows the same strategy as in the rank-$3$ case (see section 3 in \cite{Ch20}). Firstly, there is a basis such that the action of $G_F$ on $\phi[\mathfrak{l}]$ is of one of the following forms:
\begin{enumerate}
\item[(i)]$\left(\begin{array}{cc}\chi & * \\0 & B\end{array}\right)$ (if $X$ has dimension $1$),
\item[(ii)]$\left(\begin{array}{cc}B & * \\0 & \chi\end{array}\right) $ (if $X$ has codimension $1$),
\end{enumerate}
where $B:G_F \longrightarrow {\rm{GL}}_{r-1}(\mathbb{F}_{\mathfrak{l}})$ is a homomorphism  and $\chi:G_F\longrightarrow \mathbb{F}_{\mathfrak{l}}^*$ is a character. 

Now we consider the following exact sequence $$0\longrightarrow \phi[\mathfrak{l}]^{\circ} \longrightarrow \phi[\mathfrak{l}]\longrightarrow \phi[\mathfrak{l}]^{et}\longrightarrow 0$$ of $\mathbb{F}_{\mathfrak{l}}[I_\mathfrak{l}]$-modules. 

\begin{prop}\label{prop10}\rm{(\cite{PR09-2}, Proposition 2.7)}  
\begin{enumerate}
\item[i.] The inertia group $I_{\mathfrak{l}}$ acts trivially on $\phi[\mathfrak{l}]^{et}$.
\item[ii.] The $\mathbb{F}_{\mathfrak{l}}$-vector space $\phi[\mathfrak{l}]^{\circ}$ extends uniquely to a one dimensional $\mathbb{F}_\mathfrak{l}^{(h)}$-vector space structure such that the action of $I_{\mathfrak{l}}$ on $\phi[\mathfrak{l}]^\circ$ is given by the fundamental character $\zeta_n$.
\item[iii.] The action of wild inertia group at $\mathfrak{l}$ on $\phi [ \mathfrak{l}]^{\circ}$ is trivial.

\end{enumerate}

\end{prop}
\begin{cor}\label{cor11}

$\phi[\mathfrak{l}]^\circ$ is an irreducible $\mathbb{F}_{\mathfrak{l}}[I_\mathfrak{l}]$-module.

\end{cor}
\begin{proof}
This immediately follows from Proposition \ref{prop10}(ii).
\end{proof}

By Corollary \ref{cor11}, we have either  (i) $X\cap\phi[\mathfrak{l}]^{\circ}=\{0\}$ or (ii) $ X\cap\phi[\mathfrak{l}]^{\circ}=\phi[\mathfrak{l}]^{\circ}$. 

\begin{claim}
 Either $\det B$ is unramified at every prime ideal $\mathfrak{p}$ of $A$ or $\chi$ is unramified at every prime ideal $\mathfrak{p}$ of $A$. 
\end{claim}
\begin{proof}
If $\mathfrak{p}$ is a prime different from $(T)$ and $\mathfrak{l}$, then $\bar{\rho}_{\phi,\mathfrak{l}}$ is unramified at $\mathfrak{p}$. Thus $\det B$ and $\chi$ are both unramified at $\mathfrak{p}$. It suffices for us to prove the cases $\mathfrak{p}=(T)$ and $\mathfrak{p}=\mathfrak{l}$.\\
For $\mathfrak{p}=(T)$, Lemma \ref{lem22-1} implies $\#\bar{\rho}_{\phi,\mathfrak{l}}(I_T)$ is a $q$-power. This means the order of $B(\sigma)$ and $\chi(\sigma)$ are also $q$-powers. Therefore, the order of $\det B(\sigma)$ and $\chi(\sigma)$ are also $q$-powers for all $\sigma \in I_T$. However, $\det B(\sigma)$ and $\chi(\sigma)$ belong to $\mathbb{F}_{\mathfrak{l}}^*$,which is of order prime to $q$. Hence $\det B(\sigma)=\chi(\sigma)=1$ for all $\sigma \in I_T$. \\
For $\mathfrak{p}=\mathfrak{l}$, we have two cases.
\begin{enumerate}
\item[Case (i):]
$X$ has dimension $1$.  Then $X$ is an irreducible $\mathbb{F}_{\mathfrak{l}}[G_F]$-module. This means either $X=\phi[\mathfrak{l}]^\circ$ or $X\cap \phi[\mathfrak{l}]^\circ=\{0\}$. Let $\sigma\in I_{\mathfrak{l}}.$ If $X=\phi[\mathfrak{l}]^\circ$, then the matrix $B(\sigma)$ describes how $\sigma$ acts on the $\mathbb{F}_{\mathfrak{l}}[I_{\mathfrak{l}}]$-module $\phi[\mathfrak{l}]/\phi[\mathfrak{l}]^{\circ} \cong \phi[\mathfrak{l}]^{et}$. Proposition \ref{prop10} then implies $\det B(\sigma)=1$ for all $\sigma\in I_{\mathfrak{l}}$. When $X\cap \phi[\mathfrak{l}]^\circ=\{0\}$, we can view $X$ as a $\mathbb{F}_{\mathfrak{l}}[I_{\mathfrak{l}}]$-submodule of $\phi[\mathfrak{l}]/\phi[\mathfrak{l}]^{\circ} \cong \phi[\mathfrak{l}]^{et}$. Therefore, Proposition \ref{prop10} implies $\chi(\sigma)=1$ for all $\sigma\in I_{\mathfrak{l}}$.
\item[Case (ii):]
$X$ has codimension $1$.  As $\phi[\mathfrak{l}]^\circ$ is an irreducible $\mathbb{F}_{\mathfrak{l}}[I_{\mathfrak{l}}]$- module by Corollary \ref{cor11}, we get either $X\supseteq \phi[\mathfrak{l}]^\circ$ or $X\cap \phi[\mathfrak{l}]^\circ=\{0\}$. When $X\supseteq \phi[\mathfrak{l}]^\circ$, we have the ``modulo $X$'' map $\phi[\mathfrak{l}]^{et}\cong \phi[\mathfrak{l}]/\phi[\mathfrak{l}]^{\circ}\rightarrow \phi[\mathfrak{l}]/X$  as $\mathbb{F}_{\mathfrak{l}}[I_\mathfrak{l}]$-modules. Thus $\chi(\sigma)=1$ for all $\sigma\in I_{\mathfrak{l}}$ by Proposition \ref{prop10}. The argument for the case when $X\cap \phi[\mathfrak{l}]^\circ=\{0\}$ is similar to what happened in case (i), so $\det B(\sigma)=1$ for all $\sigma\in I_{\mathfrak{l}}$.
\end{enumerate}
\end{proof}
We first assume that $\det B$ is unramified at every prime $\mathfrak{p}$.
The homomorphism $\det B:G_F\rightarrow \mathbb{F}_{\mathfrak{l}}^*$ factors through ${\rm{Gal}}(K/\mathbb{F}_q(T))$, where $K$ is a finite abelian extension of $\mathbb{F}_q(T)$, unramified at every finite place and tamely ramified at infinity. Hayes \cite{Hay74} studied the class field theory in function field analogue and showed that such finite abelian extension is just some constant extension of $\mathbb{F}_q(T)$ (see section 5 and Theorem 7.1 in \cite{Hay74}), hence we have $K\subseteq \bar{\mathbb{F}}_q(T)$. Now we can write $\det B$ in the following way:

$$\det B: G_F\twoheadrightarrow {\rm{Gal}}(\bar{\mathbb{F}}_q(T)/\mathbb{F}_q(T))\xrightarrow{\sim}{\rm{Gal}}(\bar{\mathbb{F}}_q/\mathbb{F}_q)\rightarrow \mathbb{F}_{\mathfrak{l}}^*,$$
where the first map is the restriction map.

This implies there is some element $\xi \in \mathbb{F}_{\mathfrak{l}}^*$ such that $\det B({\rm{Frob}}_{\mathfrak{p}})=\xi^{\deg \mathfrak{p}}$ for every prime ideal $\mathfrak{p}$ of $A$ not equal to $(T)$ or $\mathfrak{l}$. Now we consider the Frobenius element for a degree-$1$ prime $\fp=(T-c)\neq (T) \text{ or }\fl$. The characteristic polynomial of ${\rm{Frob}}_{\fp}$ acting on $\phi[\fl]$ is $\bar{P}_{\phi,\mathfrak{p}}(x)=x^r+x^{r-1}-\bar{\mathfrak{p}}$. Thus we have the following factorization

$$x^r+x^{r-1}-\bar{\mathfrak{p}}= (x^{r-1}-\alpha_{\mathfrak{p}}x^{r-2}+\cdots+\xi)(x-\xi^{-1}{\bar{\mathfrak{p}}})\in \mathbb{F}_\mathfrak{l}[x].$$

Hence for distinct prime ideals $\mathfrak{p}_1$ and $\mathfrak{p}_2$ of degree 1, we can factorize the characteristic polynomial of $\bar{\rho}_{\phi,\mathfrak{l}}({\rm{Frob}}_{\mathfrak{p}_1})$ and $\bar{\rho}_{\phi,\mathfrak{l}}({\rm{Frob}}_{\mathfrak{p}_2})$, respectively. By computing their coefficients, we get 

$$
\left\{
\begin{array}{ccc}
\alpha_{\mathfrak{p}_1}+\xi^{-1}\bar{\mathfrak{p}}_1 & =\alpha_{\mathfrak{p}_2}+\xi^{-1}\bar{\mathfrak{p}}_2 & =-1\\
\alpha_{\mathfrak{p}_1}\xi^{-1}\bar{\mathfrak{p}}_1+\xi & =\alpha_{\mathfrak{p}_2}\xi^{-1}\bar{\mathfrak{p}}_2+\xi &=0
\end{array}
\right.
$$

$$
\Rightarrow
\left\{
\begin{array}{ccc}

\alpha_{\mathfrak{p}_1}&=&-1-\xi^{-1}\bar{\mathfrak{p}}_1\\
\alpha_{\mathfrak{p}_2}&=&-1-\xi^{-1}\bar{\mathfrak{p}}_2\\
\alpha_{\mathfrak{p}_1}\bar{\mathfrak{p}}_1&=&\alpha_{\mathfrak{p}_2}\bar{\mathfrak{p}}_2
\end{array}
\right.
$$

$$
\Rightarrow (-\xi-\bar{\mathfrak{p}}_1)\bar{\mathfrak{p}}_1=(-\xi-\bar{\mathfrak{p}}_2)\bar{\mathfrak{p}}_2
$$
$$
\Rightarrow -\xi(\bar{\mathfrak{p}}_1-\bar{\mathfrak{p}}_2)=(\bar{\mathfrak{p}}_1^2-\bar{\mathfrak{p}}_2^2).
$$
So $\xi \equiv -({\mathfrak{p}}_1+{\mathfrak{p}}_2)\ {\rm{mod}}\ \mathfrak{l}.$
As $q> 5$, there are at least three different  prime ideals of degree $1$ not equal to $(T)$ or $\mathfrak{l}$. Thus we can derive the above argument for any two such prime ideals, which implies that all these prime ideals are congurent to each other modulo $\mathfrak{l}$. This gives us a contradiction. The arguments for `` If $\chi$ is unramified at every prime $\mathfrak{p}$'' follows the same process as above.

In conclusion, we have shown that if there is an proper $G_F$-submodule $X$ of $\phi[\fl]$, then 
$$2 \leqslant{\rm{dim}}_{\F_\fl}X=d\leqslant r-2.$$ 

Therefore, we may assume in the subsection 5.1 and 5.2 that
\begin{enumerate}
\item[(a)] $2 \leqslant{\rm{dim}}_{\F_\fl}X=d\leqslant r-2.$
\item[(b)] $\deg_T\fl\geqslant 2$
\end{enumerate}
The following subsections will split into the cases (i) $X\cap\phi[\mathfrak{l}]^{\circ}=\{0\}$ and (ii) $ X\cap\phi[\mathfrak{l}]^{\circ}=\phi[\mathfrak{l}]^{\circ}$ to derive contradictions. Then the irreducibility for mod $\mathfrak{l}$ representations when $\mathfrak{l}\neq(T)$ is proved.

\subsection{The case $X\cap\phi[\mathfrak{l}]^{\circ}=\{0\}$}

If $X\cap\phi[\mathfrak{l}]^{\circ}=\{0\}$, then Proposition \ref{prop10} (i) shows the action of $I_\mathfrak{l}$ on $X$ is trivial. Now we consider the action of $I_T$ on $X$. By Lemma \ref{lem22}, there is a basis $\{w_1,w_2,\cdots,w_{r-1},z\}$ of $\phi[\mathfrak{l}]$ such that $\bar{\rho}_{\phi,\mathfrak{l}}(I_T)= \left\{\left(\begin{array}{cccc}1 &  &  & b_{1} \\ & \ddots &  & \vdots \\ &  & \ddots & b_{r-1} \\ &  &  & 1\end{array}\right),\ b_{i}\in \mathbb{F}_\mathfrak{l}\ \forall\ 1\leqslant\ i\leqslant\ r-1\right\}.$ Since $X$ is a proper submodule of $\phi[\mathfrak{l}]$, we have $X\subseteq {\rm{span}}\{w_1,w_2,\cdots,w_{r-1}\}$. Therefore, $I_T$ acts trivially on $X$. For places $\mathfrak{p}\neq(T)\  {\rm{or}}\ \mathfrak{l}$, we know that $I_\mathfrak{p}$ acts trivially on $\phi[\mathfrak{l}]$, hence $I_\mathfrak{p}$ acts trivially on $X$.

In conclusion, we have proved the following proposition:
\begin{prop}\label{prop25}
The action of $G_F$ on $X$  is unramified at every finite place. 
\end{prop}

Now we consider the ramification index at infinity of the extension $F(\phi[\mathfrak{l}])/F$.
Let $F_\infty=\mathbb{F}_q((\frac{1}{T}))$ be the local field of $F$ at $\infty$ with valuation $v_\infty$, and $|\cdot|_\infty$ be its corresponding absolute value. Let $\Lambda$ be the period lattice of the polynomial $\phi_T(x)=Tx+x^{q^{r-1}}+T^{q-1}x^{q^r}.$ From section 1.1 of \cite{Gek19}, we know $\Lambda$ is a discrete free $A$-submodule of $\hat{\bar{F}}_\infty$, the completed algebraic closure of $F_\infty$, of rank $r$. We also have 
$$e_{\Lambda}(Tz)=\phi_T(e_\Lambda(z))\ \ \ {\rm{for\ all\ }} z\in\hat{\bar{F}}_\infty,$$
where
$$e_\Lambda(z)=z\prod_{0\neq\lambda\in\Lambda}(1-z/\lambda).$$

Let $F_\infty(\Lambda)$ be the field extension generated by the period lattice $\Lambda$. From Proposition 1.2 of \cite{Gek19}, we have
$$F_\infty(\Lambda)=F_\infty({\rm{tor}}(\phi)),$$
where ${\rm{tor}}(\phi)=\bigcup_{0\neq a\in A}\phi[a]$. Let $\{\lambda_1,\lambda_2,\cdots,\lambda_r\}$ be a successive minimum basis of the lattice $\Lambda$ (see \cite{Gek19}, 1.3 for the definition). By Proposition 1.4 (i) of \cite{Gek19}, the spectrum $(|\lambda_1|_\infty,|\lambda_2|_\infty,\cdots,|\lambda_r|_\infty)$ is related to the Newton polygon of $\phi_T(x)$ with respect to the valuation $v_\infty$. The Newton polygon of $\phi_T(x)$ is a line segment with slope $\frac{2-q}{q^r-1}$. Hence, by Proposition 1.4(ii) in \cite{Gek19}, we have
$$|\lambda_1|_\infty=|\lambda_2|_\infty=\cdots=|\lambda_r|_\infty.$$

Let $\mu_i=e_\Lambda(\frac{\lambda_i}{T})$, the set $\{\mu_1,\mu_2,\cdots,\mu_r\}$ is a $\mathbb{F}_q$-basis of $\phi[T]$. By formula 1.3.2 of \cite{Gek19}, we have
$$|\mu_i|_\infty=|e_\Lambda(\frac{\lambda_i}{T})|_\infty=|\frac{\lambda_i}{T}|_\infty.$$
From the Newton polygon of $\phi_T(x)$, we can compute $v_\infty(\mu_i)=\frac{q-2}{q^r-1}$. 
Thus we can deduce 
$$v_\infty(\lambda_i)=v_\infty(\mu_i)-1=\frac{-q^r+q-1}{q^r-1}.$$
Now we can study the extension $F_{\infty}(\Lambda)/F_\infty$. We use the notation set up in \cite{Gek19}, section 2.1. In our case, $L=F_\infty$, $\mathbf{B}_1=\{\lambda_1,\lambda_2,\cdots,\lambda_r\}$, $\tau=t=1$, $L_\tau=F_\infty(\Lambda)$. By applying diagram 2.8.4 of \cite{Gek19} to our situation, we have the followings:
\begin{enumerate}
 \item $\tilde{L}_\tau=L_0(V_\tau)=F_\infty(\Lambda)$ 
 \item $L_{\tau-1}=L_0=F_\infty$
\item $V_\tau=\text{the $\mathbb{F}_q$-vector space generated by basis $e_0(B_1)=\{\lambda_1, \cdots, \lambda_r \}$}$ (see \cite{Gek19}, 2.4)
\item $V_\tau$ is pure with precise denominator of weight equal to $q^r-1$  by our computation on $v_\infty(\lambda_i)$, Definition 2.5 and Proposition 2.6 in \cite{Gek19}. 
\item $L'_{\tau-1}= \text{a completely ramified separable extension of $L_{\tau-1}$ of degree $q^r-1$}$. (see \cite{Gek19}, 2.8)
\item $M_\tau=L_{\tau-1}(V_\tau)\cap L'_{\tau-1}=F_\infty(\Lambda)\cap L'_{\tau-1}$, and $[M_\tau:L_{\tau-1}]=e(\tilde{L}_\tau| L_{\tau-1})=e(F_\infty(\Lambda)| F_\infty)$ (see \cite{Gek19}, 2.9.1)
\end{enumerate}
Therefore, $[M_\tau:L_{\tau-1}]$ must divide $q^r-1$. This implies the ramification index $e(F_\infty(\Lambda)| F_\infty)$ is some number prime to $p$. 
Hence the field extension $F_\infty({\rm{tor}}(\phi))/F_\infty$ is tamely ramified, which means the field extension $F({\rm{tor}}(\phi))/F$ is tamely ramified at $\infty$. Since $F(\phi[\mathfrak{l}])$ is an intermediate field of the extension $F({\rm{tor}}(\phi))/F$, the extension $F(\phi[\mathfrak{l}])/F$ is tamely ramified at $\infty$.
Combining with Proposition \ref{prop25}, we can make the following statement:
\begin{enumerate}

\item[($\ast$).]
The action of $G_F$ on $X$  is unramified at every finite place and tamely ramified at the place of infinity. 
\end{enumerate}

Finally, we claim that the field extension $\bar{\F}_qF(X)/\bar{\F}_qF$ is nontrivial. This leads to a non-trivial Galois cover of $\mathbb{P}^{1}_{{\bar{\mathbb{F}}_q}}$ which is unramified away from infinity and tamely ramified at infinity, hence is a contradiction by Hurwitz genus formula. 

Suppose the field extension $\bar{\F}_qF(X)/\bar{\F}_qF$ is trivial. Then for a element $0\neq w\in X\subset \bar{\F}_q(T)$, the valuation $v_T(w)$ is an integer. We may write 

$$w=\sum_{i=-m}^{n}c_iT^i, \text{ where $m, n\in \mathbb{Z}$, and $c_i\in\bar{\F}_q$ with $c_{-m}, c_{n}\in \bar{\F}_q^*$}.$$

We show that $w$ must lie in $\bar{\F}_q[T]$. Suppose $m\geqslant 1$, we have the valuation $v_T(w)=-m$. Now we compute the valuation of $\phi_T(w)=Tw+w^{q^{r-1}}+T^{q-1}w^{q^r}$. As $v_T(w^{q^{r-1}}) = q^{r-1}\cdot(-m)$ and $v_T(T^{q-1}w^{q^r})=q^r\cdot(-m)+(q-1)$, we have
$$q^r\cdot(-m)+(q-1)=v_T(\phi_T(w))<v_T(w)=-m.$$
Therefore, we can see that
$$v_T(\phi_{T^j}(w))<v_T(\phi_{T^{j-1}}(w))\text{ for any integer $j\geqslant 1$}.$$
Hence the term of $\phi_{T^{\deg_T(l)}}(w)$ with  smallest valuation  does  not appear in $\phi_{T^j}(w)$ for any integer $0\leqslant j< \deg_T(l)$. We can deduce from this that the term of $\phi_\fl(w)$ with smallest valuation cannot be eliminated, which contradicts to $w\in X\subset \phi[\fl]$. Thus $w$ must lie in $\bar{\F}_q[T]$. 

On the other hand, we can replace $v_T$ by $v_\infty$ and run the same process as above to show that $w$ must lie in $\bar{\F}_q^*$. Indeed, we assume $v_\infty(w)=-m\leqslant -1$. We can deduce that 
$$q^r\cdot(-m)-(q-1)=v_\infty(\phi_T(w))<v_\infty(w)=-m.$$
Hence we have $v_\infty(\phi_{T^j}(w))<v_\infty(\phi_{T^{j-1}}(w))\text{ for any integer $j\geqslant 1$}.$ The remains follow the same argument as above.

Now we have the element $0\neq w\in X\subset \phi[\fl]$ lies in $\bar{\F}_q^*$. However, this implies the term of $\phi_\fl(w)$ with largest $T$-degree is the leading term
$$T^{(q-1)( {\displaystyle  \sum_{i=1}^{\deg_T(l)}q^{r(i-1)} })}\cdot w^{q^{r\cdot \deg_T(l)}}.$$
Moreover, this term in $\phi_\fl(w)$ cannot be eliminated since the $T$-degree of the leading coefficient in $\phi_\fl(x)$ is strictly larger than the $T$-degree of any other coefficient. Thus we have a contradiction due to the fact that $w\in X\subset \phi[\fl]$. In conclusion, the field extension $\bar{\F}_qF(X)/\bar{\F}_qF$ is nontrivial.

\subsection{The case $ X\cap\phi[\mathfrak{l}]^{\circ}=\phi[\mathfrak{l}]^{\circ}$}

If $X\cap\phi[\mathfrak{l}]^{\circ}=\phi[\mathfrak{l}]^{\circ}$, then we have $\phi[\mathfrak{l}]^{\circ}\subseteq X\subsetneq \phi[\mathfrak{l}]$. We consider the following isogeny over $F$:
$$u:\phi \rightarrow \psi:=\phi/X.$$

Now we can study the action of $G_F$ on the proper $G_F$-submodule $u(\phi[\mathfrak{l}])$ of $\psi[\mathfrak{l}]$. We claim that the extension $F(u(\phi[\fl]))/F$ is unramified at every finite place and tamely ramified at the place of infinity.

The isogeny induces an $G_F$-equivariant isomorphism between rational Tate modules
$$u:V_\mathfrak{l}(\phi)\rightarrow V_\mathfrak{l}(\psi),$$
where $V_\mathfrak{l}(\phi):=T_\mathfrak{l}(\phi)\otimes F_\mathfrak{l}$. For any prime ideal $\mathfrak{p}\neq (T) \text{ or } \mathfrak{l}$. Since $\phi$ has good reduction at $\fp$ and $u$ is $G_F$-equivariant, we know $\psi$ has good reduction at $\mathfrak{p}$. Therefore, $G_F$ acts on $u(\phi[\mathfrak{l}])$ is unramified at $\mathfrak{p}\neq (T) \text{ or } \mathfrak{l}$.

For the prime $\mathfrak{p}=\mathfrak{l}$, the action of inertia group $I_\mathfrak{l}$ on $u(\phi[\mathfrak{l}])$ is determined by the action of $I_\mathfrak{l}$ on $\phi[\mathfrak{l}]/X$. Combining the condition $X\supseteq \phi[\mathfrak{l}]^{\circ}$ and Proposition \ref{prop10} (ii), we deduce that $I_\mathfrak{l}$ acts trivially on $\phi[\mathfrak{l}]/X$. Hence $I_\mathfrak{l}$ acts trivially on $u(\phi[\mathfrak{l}])$ as well.

For the place $\fp=\infty$, we proved the extension $F({\rm{tor}}(\phi))/F$ is tamely ramified at $\infty$ in the end of previous subsection. This implies the wild inertia group at $\infty$ acts trivially on the Tate module $T_\fl(\phi)$, hence the wild inertia group acts trivially on $V_\fl(\phi)$ as well. Now the above $G_F$-equivariant isomorphism $V_\fl(u)$ implies the wild inertia at $\infty$ acts trivially on $V_\fl(\psi)$. Therefore, we can deduce that the wild inertia group at $\infty$ acts trivially on $u(\phi[\fl])$. In other words, the extension $F(u(\phi[\fl]))/F$ is tamely ramified at infinity.

Finally, for the prime $\mathfrak{p}=(T)$, from Proposition \ref{red} and Lemma \ref{lem22-1} we know that $\psi$ has stable bad reduction at $\fp$.

Once $\psi$ has stable bad reduction at $\fp$, we know that $\psi$ has stable bad reduction at $\fp$ of rank $r-1$. Moreover, both $u$ and $\psi$ can be defined over $A_\fp$. As ${\rm{dim}}_{\F_\fl}X=d$, then the $\tau$-degree of $u$ is $d$. We may write
$$u=a_0+a_1\tau+\cdots+a_d\tau^d\text{ and }\psi_T=T+g_1\tau+\cdots+g_r\tau^r,$$
where $a_i$ and $g_j$ are elements in $A_\fp$ with $a_d$, $g_r\in A_\fp^*$. By comparing the leading coefficients of $u\phi_T=\psi_Tu$, we get 
$$T^{q^d(q-1)}=g_ra_d^{q^r-1}.$$
Therefore, we can conclude $g_r=T^{q^d(q-1)}\cdot m$ for some $m\in A_\fp^*$.

Now we can follow the same process as in Lemma \ref{lem22-1} to deduce the following:

There is a basis $\{w_1\cdots, w_{r-1}, z\}$ of $\psi[\mathfrak{l}]$ such that 
$$\bar{\rho}_{\psi,\mathfrak{l}}(I_\fp) \subseteq \left\{\left(\begin{array}{cccc}1 &  &  & b_{1} \\ & \ddots &  & \vdots \\ &  & \ddots & b_{r-1} \\ &  &  & 1\end{array}\right),\ b_{i}\in \mathbb{F}_\mathfrak{l}\ \forall\ 1\leqslant\ i\leqslant\ r-1\right\}.$$ with respect to the basis.

Furthermore, we can give an estimation on the size of $\bar{\rho}_{\psi,\mathfrak{l}}(I_\fp)$ by using the same process as in Lemma \ref{lem22}. The computation only involves in the leading coefficient $g_r$ of $\psi_T$. Because $g_r=T^{q^d(q-1)}\cdot m$ for some $m\in A_\fp^*$, the left hand side of equation (1) in the proof of Lemma \ref{lem22} becomes $$T^{(q-1)q^d( {\displaystyle  \sum_{i=1}^{\deg_T(\mathfrak{l})}q^{r(i-1)} })}.$$ Hence we can deduce the estimation

$$|\bar{\rho}_{\psi,\mathfrak{l}}(I_\fp)|\geqslant q^{(r-1)\deg_T\fl-d}.$$

\begin{claim}

$I_\fp$ acts trivially on $u(\phi[\fl])\subset \psi[\fl]$.

\end{claim}
\begin{proof}[Proof of claim]

We split ${\rm{dim}}_{\F_\fl}X=d$ into two cases.
\begin{enumerate}

\item[$d\leqslant\deg_T(\fl)$:] In this case, we can deduce from the estimation $|\bar{\rho}_{\psi,\mathfrak{l}}(I_\fp)|\geqslant q^{(r-1)\deg_T\fl-d}$ that

$$|\bar{\rho}_{\psi,\mathfrak{l}}(I_\fp)|\geqslant q^{(r-2)\deg_T\fl}.$$

As ${\rm{dim}}_{\F_\fl}u(\phi[\fl])=r-d$ and $1 < d < r-1$, we know ${\rm{dim}}_{\F_\fl}u(\phi[\fl])\leqslant r-2$. Now we use the basis $\{w_1,\cdots, w_{r-1}, z\}$ to present vectors in the $\F_\fl$-space $u(\phi[\fl])$. Suppose there is an element $w\in u(\phi[\fl])$ with linear combination $c_1w_1+\cdots+c_{r-1}w_{r-1}+cz$ where $c_i\in \F_\fl$ and $c\in\F_\fl^*$. Since we have
$$\bar{\rho}_{\psi,\mathfrak{l}}(I_\fp) \subseteq \left\{\left(\begin{array}{cccc}1 &  &  & b_{1} \\ & \ddots &  & \vdots \\ &  & \ddots & b_{r-1} \\ &  &  & 1\end{array}\right),\ b_{i}\in \mathbb{F}_\mathfrak{l}\ \forall\ 1\leqslant\ i\leqslant\ r-1\right\}.$$
For every $\sigma=\left(\begin{array}{cccc}1 &  &  & b_{1,\sigma} \\ & \ddots &  & \vdots \\ &  & \ddots & b_{r-1,\sigma} \\ &  &  & 1\end{array}\right) \in \bar{\rho}_{\psi,\mathfrak{l}}(I_\fp)$,  we consider $$\sigma\cdot w-w=b_{1,\sigma}\cdot w_1+\cdots+b_{r-1,\sigma}\cdot w_{r-1}.$$ We may start with an element $\sigma_1\in \bar{\rho}_{\psi,\mathfrak{l}}(I_\fp)$ with $b_{1,\sigma_1},\cdots, b_{r-1,\sigma_1}$ not all zero. Hence there are at most $|\F_\fl|=q^{\deg_T\fl}$ many elements $\sigma\in \bar{\rho}_{\psi,\mathfrak{l}}(I_\fp)$ such that $\sigma\cdot w-w$ is linearly dependent with $\sigma_1\cdot w-w$. From our estimation $|\bar{\rho}_{\psi,\mathfrak{l}}(I_\fp)|\geqslant q^{(r-2)\deg_T\fl}$, we can find $\sigma_2\in \bar{\rho}_{\psi,\mathfrak{l}}(I_\fp)$ such that $\sigma_2\cdot w-w$ is linearly independent with $\sigma_1\cdot w-w$. Now there are at most $|\F_\fl|^2=q^{2\deg_T\fl}$ many elements $\sigma\in \bar{\rho}_{\psi,\mathfrak{l}}(I_\fp)$ such that $\sigma\cdot w-w$ is linearly dependent with $\{\sigma_1\cdot w-w, \sigma_2\cdot w-w\}$. From the estimation again, we can find $\sigma_3\in \bar{\rho}_{\psi,\mathfrak{l}}(I_\fp)$ such that $\{\sigma_1\cdot w-w, \sigma_2\cdot w-w, \sigma_3\cdot w-w\}$ is a linearly independent set. By following this procedure, we can pick $\sigma_1, \sigma_2, \cdots, \sigma_{r-2}\in \bar{\rho}_{\psi,\mathfrak{l}}(I_\fp)$ such that $$\{\sigma_i\cdot w-w\mid 1\leqslant i\leqslant r-2\}$$ is a linearly independent set. Therefore, $\sigma_i\cdot w-w$ together with $w$ produce $r-1$ many linearly independent vectors in $u(\phi[\fl])$, a contradiction. Thus all elements in $u(\phi[\fl])$ are linear combinations of the basis vectors $w_1, w_2, \cdots, w_{r-1}$, which implies that $I_\fp$ acts trivially on $u(\phi[\fl])$.

\item[$d>\deg_T(\fl)$:] In this case, we are in the situation where
$$1<\deg_T(\fl)<d\leqslant r-2.$$
Therefore, the estimation $|\bar{\rho}_{\psi,\mathfrak{l}}(I_\fp)|\geqslant q^{(r-1)\deg_T\fl-d}$ implies

$$|\bar{\rho}_{\psi,\mathfrak{l}}(I_\fp)|\geqslant q^{(r-1)\deg_T\fl-d}>q^{(r-1-d)\deg_T\fl}.$$

As ${\rm{dim}}_{\F_\fl}u(\phi[\fl])=r-d$, we use the basis $\{w_1,\cdots, w_{r-1}, z\}$ to present vectors in the $\F_\fl$-space $u(\phi[\fl])$ again. Again, we suppose there is an element $w\in u(\phi[\fl])$ with linear combination $c_1w_1+\cdots+c_{r-1}w_{r-1}+cz$ where $c_i\in \F_\fl$ and $c\in\F_\fl^*$. For every $\sigma\in \bar{\rho}_{\psi,\mathfrak{l}}(I_\fp)$,  we consider $$\sigma\cdot w-w=b_{1,\sigma}\cdot w_1+\cdots+b_{r-1,\sigma}\cdot w_{r-1}.$$
By the procedure in the previous case and our estimation $|\bar{\rho}_{\psi,\mathfrak{l}}(I_\fp)|>q^{(r-1-d)\deg_T\fl}$, we can find $r-d$ many elements $\sigma_1,\cdots, \sigma_{r-d}\in \bar{\rho}_{\psi,\mathfrak{l}}(I_\fp)$ such that the set 
$$\{\sigma_i\cdot w-w\mid 1\leqslant i\leqslant r-d\}$$
is linearly independent.
Thus $\sigma_i\cdot w-w$ together with $w$ produce $r-d+1$ many linearly independent vectors in $u(\phi[\fl])$, which is a contradiction because ${\rm{dim}}_{\F_\fl}u(\phi[\fl])=r-d$. Therefore, all elements in $u(\phi[\fl])$ are linear combinations of the basis vectors $w_1, w_2, \cdots, w_{r-1}$. Hence $I_\fp$ acts trivially on $u(\phi[\fl])$.

\end{enumerate}

\end{proof}

From our study of the ramifications of  $G_F$-action on $u(\phi[\fl])$, we conclude
\begin{lemma}
The action of $G_F$ on $u(\phi[\fl])$  is unramified at every finite place and tamely ramified at the place of infinity. 
\end{lemma}

Finally, we prove the field extension $\bar{\F}_qF(u(\phi[\fl]))/\bar{F}_qF$ is nontrivial. This leads to a non-trivial Galois cover of $\mathbb{P}^{1}_{{\bar{\mathbb{F}}_q}}$ which is unramified away from infinity and tamely ramified at infinity, which is a contradiction by Hurwitz genus formula. 

Suppose the field extension $\bar{\F}_qF(u(\phi[\fl]))/\bar{F}_qF$ is trivial, then for any element $\alpha\in\phi[\fl]$ the valuation $v_\infty(u(\alpha))$ is an integer. On the other hand, we take the minimum successive basis $\{\lambda_1,\cdots, \lambda_r\}$ from the previous subsection and set $\alpha_i=e_\Lambda(\frac{\lambda_i}{l})$. The set $\{\alpha_1,\cdots, \alpha_r\}$ is a $\F_\fl$-basis of $\phi[\fl]$. Moreover, we have 
$$v_\infty(\alpha_i)=v_\infty(e_\Lambda(\frac{\lambda_i}{l}))=v_\infty(\frac{\lambda_i}{l})=\deg_T(l)+\frac{-q^r+q-1}{q^r-1}.$$ Now we compute the valuation $v_\infty(u(\alpha_i))$. As $u=a_0+a_1\tau+\cdots+a_d\tau^d\in F\{\tau\}$, we have
$$u(\alpha_i)=a_0\alpha_i+a_1\alpha_i^q+\cdots+a_d\alpha_i^{q^d}.$$
For each nonzero term $a_j\alpha_i^{q^j}$ of $u(\alpha_i)$ where $0\leqslant j \leqslant d< r-1$, its valuation is
$$v_\infty(a_j\alpha_i^{q^j})=v_\infty(a_j)+q^j\cdot v_\infty(\alpha_i)=\text{some integer}+q^j\cdot \frac{-q^r+q-1}{q^r-1}=\text{some integer}+\frac{q^j\cdot(q-2)}{q^r-1}.$$
Therefore, the valuation of each nonzero term of $u(\alpha_i)$ has distinct fractional part. Thus $v_\infty(u(\alpha_i))$ is not an integer by the strong triangle inequality  because the valuation of each summand $v_\infty(a_j\alpha_i^{q^j})$ is a distinct fraction. This shows the field extension $\bar{\F}_qF(u(\phi[\fl]))/\bar{F}_qF$ is nontrivial.

In conclusion, we have proved the irreducibility of the mod $\fl$ Galois representations for $\fl\neq (T)$. Combining with Theorem \ref{surjmodl}, we have the following corollary:

\begin{cor}\label{cor26}
Let $q=p^e$ be a prime power, $A=\mathbb{F}_q[T]$, and $F=\mathbb{F}_q(T)$. Assume $r\geqslant 3$ is a prime number, there is a constant $c=c(r)\in \mathbb{N}$ depending only on $r$ such that for $p>c(r)$ the following statement is true ${\rm{:}}$

Let $\phi$ be a Drinfeld $A$-module over $F$ of rank $r$ with generic characteristic, which  is defined by $\phi_T=T+\tau^{r-1}+T^{q-1}\tau^r$. Let $\mathfrak{l}\neq (T)$ be a finite place of $F$. Then the mod $\mathfrak{l}$ representation 
$$\bar{\rho}_{\phi,\mathfrak{l}}:G_F\longrightarrow {\rm{Aut}}(\phi[\mathfrak{l}])\cong {\rm{GL}}_r(\mathbb{F}_{\mathfrak{l}})$$
 is surjective.
 \end{cor}
 
 \subsection{The case when $\mathfrak{l}=(T)$}
 For the mod $(T)$ Galois representation $\bar{{\rho}}_{\phi,T}: G_F \longrightarrow {\rm{Aut}}(\phi[T])\cong {\rm{GL_r}}(\mathbb{F}_q)$, we are actually computing the Galois group ${\rm{Gal}}(F(\phi[T])/F)$ of the field extension obtained by adjoining the roots of $\phi_{T}(x)=Tx+x^{q^{r-1}}+T^{q-1}x^{q^r}$ to $F$. This question has been studied by Abhyankar \cite{Abh94}.  Here we use the same notation as in \cite{Abh94}, Theorem 3.2 (3.2.3).
 
 Let $k_0=\mathbb{F}_q$, $K=k_0(\frac{1}{T})$, and 
 $$V=V(Y)=Y^{q^r-1}+\frac{1}{T^{q-1}}Y^{q^{r-1}-1}+\frac{1}{T^{q-2}}.$$ 
We have $\mu=r-1$ and ${\rm{GCD}}(\nu,\tau)=1$, where $\nu=1+q+\cdots+q^{r-1}$ and $\tau=1+q+\cdots+q^{r-2}$. Further, we have $C_{1}=(\frac{1}{T})^{q-1}$ and $C_r=(\frac{1}{T})^{q-2}$. Thus $\rho=q-1$ and $\sigma=q-2$. Certainly we have $\rho\neq \frac{\sigma(\nu-\tau)}{\nu}$. Thus we can apply \cite{Abh94}, Theorem 3.2 (3.2.3). Since ${\rm{GCD}}(\sigma,q-1)=1$, the theorem implies 
$${\rm{Gal}}({\mathbb{F}}_q(T)(\phi[T])/{\mathbb{F}}_q(T))={\rm{GL}}_r(\mathbb{F}_q).$$

As a summary of section 5, we have proved the surjectivity of mod $\fl$ Galois representation $\bar{\rho}_{\phi,\fl}$ for any prime ideal $\fl$ of $A$ under the assumption on $q$ in Theorem \ref{surjmodl}

\section{Surjectivity of $\mathfrak{l}$-adic Galois representations}

Similar to the rank $3$ case, we wish to apply \cite{PR09}, Proposition 4.1 to prove surjectivity of $\mathfrak{l}$-adic representations. We separate $\mathfrak{l}$ into two cases:
\begin{enumerate}
\item[Case 1.] 
$\mathfrak{l}\neq (T)$\\
Our proof of the equality
$\bar{\rho}_{\phi,\mathfrak{l}}(I_T)= \left\{\left(\begin{array}{cccc}1 &  &  & b_{1} \\ & \ddots &  & \vdots \\ &  & \ddots & b_{r-1} \\ &  &  & 1\end{array}\right),\ b_{i}\in \mathbb{F}_\mathfrak{l}\ \forall\ 1\leqslant\ i\leqslant\ r-1\right\}$ only restrict $\mathfrak{l}$ to be prime to $(T)$. Hence we can prove the above equality for mod $\mathfrak{l}^2$ representation. In other words, we have 

$$\bar{\rho}_{\phi,\mathfrak{l}^2}(I_T)= \left\{\left(\begin{array}{cccc}1 &  &  & b_{1} \\ & \ddots &  & \vdots \\ &  & \ddots & b_{r-1} \\ &  &  & 1\end{array}\right),\ b_{i}\in (A/\mathfrak{l}^2)\ \forall\ 1\leqslant\ i\leqslant\ r-1\right\}$$
Therefore, the mod $\mathfrak{l}^2$ representation certainly contains a non-scalar matrix that becomes identity after modulo $\mathfrak{l}$.
\item[Case 2.]$\mathfrak{l}=(T)$\\
We consider the following diagram and focus on the representations of decomposition subgroups.
$$
\begin{array}{cccc}
\bar{\rho}_{\phi,\mathfrak{l}^2}(G_F) &\xrightarrow{\rm{modulo}\ \mathfrak{l}}& \bar{\rho}_{\phi,\mathfrak{l}}(G_F)&={\rm{GL}}_r(\mathbb{F}_\mathfrak{l}).\\
\cup&&\cup&\\
\bar{\rho}_{\phi,\mathfrak{l}^2}(G_{F_{(T)}}) &\xrightarrow{\rm{modulo}\ \mathfrak{l}}& \bar{\rho}_{\phi,\mathfrak{l}}(G_{F_{(T)}})&
\end{array}
$$

It's clear that the modulo $\mathfrak{l}$ homomorphism $\bar{\rho}_{\phi,\mathfrak{l}^2}(G_{F_{(T)}}) \xrightarrow{\rm{modulo}\ \mathfrak{l}} \bar{\rho}_{\phi,\mathfrak{l}}(G_{F_{(T)}})$ is surjective. Suppose it is injective, then we can conclude that the splitting field of $\phi_{\mathfrak{l}^2}(x)$ and of $\phi_{\mathfrak{l}}(x)$ are isomorphic over $F_{(T)}$. Now we compare the Newton's polygons of $\phi_{\mathfrak{l}^2}(x)/x$ and $\phi_{\mathfrak{l}}(x)$. 

$$
\begin{array}{ccc}
\phi_{T^2}&=&T^2+T(T^{q^{r-1}-1}+1)\tau^{r-1}+T^q(T^{q^r-1}+1)\tau^r+\tau^{2r-2}\\
                &  &\ \ \ \ \ \ +T^{q-1}(T^{(q-1)(q^{r-1}-1)}+1)\tau^{2r-1}+T^{(q-1)(q^r+1)}\tau^{2r}.
\end{array}
$$

There is some root of $ \phi_{\mathfrak{l}^2}(x)/x$ has valuation equal to $-\frac{1}{q^{2r-2}}$, but there is no element in the splitting field of $\phi_{\mathfrak{l}}(x)$ with the same valuation. Thus the splitting fields of both polynomials are not isomorphic, which implies $\bar{\rho}_{\phi,\mathfrak{l}^2}(G_{F_(T)}) \xrightarrow{\rm{modulo}\ \mathfrak{l}} \bar{\rho}_{\phi,\mathfrak{l}}(G_{F_{(T)}})$ has nontrivial kernel.

Now we prove such a nontrivial element $\bar{\rho}_{\phi,\mathfrak{l}^2}(\sigma)\in \bar{\rho}_{\phi,\mathfrak{l}^2}(G_{F_{(T)}})$ can not be a scalar matrix. Suppose $\bar{\rho}_{\phi,\mathfrak{l}^2}(\sigma)\in \bar{\rho}_{\phi,\mathfrak{l}^2}(G_{F_{(T)}})$ is a scalar matrix, then there is some $a\in \mathbb{F}_q^*$ such that $\sigma$ maps every root $\alpha$ of $\phi_{T^2}(x)/x$ to $\phi_{1+Ta}(\alpha)$. Furthermore, as $\sigma$ lies in the decomposition subgroup  $G_{F_{(T)}}$, $\sigma$ should preserve the valuation $v_T$. We pick a root $\alpha$ with valuation equal to $-\frac{1}{q^{2r-2}}$, then compare the valuation of $\sigma$ with $\sigma(\alpha)$.
$$-\frac{1}{q^{2r-2}}=v_T(\alpha)\neq v_T(\sigma(\alpha))=v_T(\phi_{1+Ta}(\alpha))=v_T([1+Ta]\alpha+a\alpha^{q^{r-1}}-[aT^{q-1}]\alpha^{q^r})=-\frac{1}{q^{r-1}}.$$
Therefore $\sigma$ cannot be a scalar matrix, and the mod $(T^2)$ representation contains a non-scalar matrix that becomes identity after modulo $\mathfrak{l}$.
\end{enumerate}
Now we can apply \cite{PR09}, Proposition 4.1 to show that the $\mathfrak{l}$-adic representation is surjective for every prime ideal $\mathfrak{l}$ of $A$.

\section{Adelic surjectivity of Galois representations}
In this section, we make a further assumption that $q\equiv 1 \mod r$. The proof of adelic surjectivity is similar to the proof for rank $2$ and $3$ cases in \cite{Zy11} and \cite{Ch20}.

\begin{lem}\label{lem28}
For each finite place $\mathfrak{l}$ of $F$, ${\rm{SL}}_r(A_\mathfrak{l})$ is equal its commutator subgroup. The only normal subgroup of ${\rm{SL}}_r(A_\mathfrak{l})$ with simple quotient is 
$$N:=\left\{B\in{\rm{SL}}_r(A_\mathfrak{l})| B\equiv\delta \cdot I_r \mod\mathfrak{l}, \ {\rm{where}}\ \delta\in \mathbb{F}_{\mathfrak{l}}\ {\rm{satisfies}}\ \delta^r=1              \right\}.$$
\end{lem}
\begin{proof}

Let $H$ be the commutator subgroup of ${\rm{SL}}_r(A_\mathfrak{l})$. It's a closed normal subgroup of 
${\rm{SL}}_r(A_\mathfrak{l})$ and ${\rm{GL}}_r(A_\mathfrak{l})$. We define $S^0:={\rm{SL}}_r(A_\mathfrak{l})$ and for $i\geqslant1$, we set 
$S^i:=\{s\in{\rm{SL}}_r(A_\mathfrak{l})\ |\ s\equiv 1\mod \mathfrak{l}^i    \}.$ For $i\geqslant 0$, define $H^i=H\cap S^i$. For $i\geqslant 0$, we define $S^{[i]}:=S^i/S^{i+1}$ and $H^{[i]}:=H^i/H^{i+1}$. There is a natural injection form $H^{[i]}$ to $S^{[i]}$ and we claim that $H^{[i]}=S^{[i]}$ for all $i\geqslant 0$.

For $i=0$, modulo $\mathfrak{l}$ induces an isomorphism $S^{[0]}\xrightarrow{\sim} {\rm{SL}}_r(\mathbb{F}_\mathfrak{l})$ and the image of $H^{[0]}$ under this isomorphism becomes the commutator subgroup of ${\rm{SL}}_r(\mathbb{F}_\mathfrak{l})$. It's well known that ${\rm{SL}}_r(\mathbb{F}_\mathfrak{l})$ is quasi-simple, i.e. ${\rm{SL}}_r(\mathbb{F}_\mathfrak{l})$ equals to its commutator subgroup and the quotient of ${\rm{SL}}_r(\mathbb{F}_\mathfrak{l})$ by its center $Z({\rm{SL}}_r(\mathbb{F}_\mathfrak{l}))$ is a simple group. Therefore, we have $H^{[0]}=S^{[0]}$.

Now we fix $i\geqslant 1$, let $\mathfrak{sl}_r(\mathbb{F}_\mathfrak{l})$ be the additive subgroup in $M_r(\mathbb{F}_\mathfrak{l})$ consisting of matrices with trace $0$. We have the following isomorphism:
$$
\begin{array}{ccc}
S^{[i]}&\xrightarrow{\sim}&\mathfrak{sl}_r(\mathbb{F}_\mathfrak{l})\\
\left[1+l^iy\right] & \mapsto & [y]
\end{array},
$$
 where $l$ is the monic polynomial in $A$ that generates $\mathfrak{l}$. Consider ${\rm{GL}}_r(A_\mathfrak{l})$ acting on both sides via conjugation action, it factors through ${\rm{GL}}_r(\mathbb{F}_\mathfrak{l})$. By \cite{PR09} Proposition 2.1, $\mathfrak{sl}_r(\mathbb{F}_\mathfrak{l})$ is an irreducible ${\rm{GL}}_r(\mathbb{F}_\mathfrak{l})$-module (here we used the assumption $q=p^e$ is a prime power with $p\geqslant5$). On the other hand, $H^{[i]}$ injects into $S^{[i]}$ and $H$ is normal in ${\rm{GL}}_r(A_\mathfrak{l})$ imply that $H^{[i]}$ is also stable under ${\rm{GL}}_r(\mathbb{F}_\mathfrak{l})$-action. Once we can show that $H^{[i]}$ is nontrivial, we have $H^{[i]}=S^{[i]}$ for all $i\geqslant0$ and so $H=S^0$. Consider the commutator map $S^0\times S^i\rightarrow H^i$ that maps $(g,h)$ to $ghg^{-1}h^{-1}$. This induces a map $S^{[0]}\times S^{[i]}\rightarrow H^{[i]}$. Combining with the isomorphism $S^{[i]} \xrightarrow{\sim} \mathfrak{sl}_r(\mathbb{F}_\mathfrak{l})$, we obtain the following map:
$${\rm{SL}}_r(\mathbb{F}_\mathfrak{l})\times \mathfrak{sl}_r(\mathbb{F}_\mathfrak{l})\rightarrow \mathfrak{sl}_r(\mathbb{F}_\mathfrak{l}),\ \ \ \ (s,X)\mapsto sXs^{-1}-X.$$
This map is not a zero map, so we have $H^{[i]}=S^{[i]}$ for all $i\geqslant0$ and $H=S^0$.

Now let $N'$ be a normal subgroup of ${\rm{SL}}_r(A_\mathfrak{l})$ with simple quotient. We consider the subgroup $S^1\triangleleft {\rm{SL}}_r(A_\mathfrak{l})$. Note that $S^1$ is a pro-$p$ group by the definition of $S^i$ and $S^0/S^1={\rm{SL}}_r(\mathbb{F}_\mathfrak{l})$ is quasi-simple. Thus there is a composition series of ${\rm{SL}}_r(A_\mathfrak{l})$ and its composition factors are ${\rm{SL}}_r(\mathbb{F}_\mathfrak{l})/Z({\rm{SL}}_r(\mathbb{F}_\mathfrak{l})),\ \mathbb{Z}/p\mathbb{Z}$ (comes from the composition factors of $S^1$) and $\mathbb{Z}/r\mathbb{Z}$ (comes from composition factors of $Z({\rm{SL}}_r(\mathbb{F}_\mathfrak{l}))$ if it is nontrivial). As ${\rm{SL}}_r(A_\mathfrak{l})$ equals to its commutator, it has no abelian quotient. Thus ${\rm{SL}}_r(A_\mathfrak{l})/N'\cong {\rm{SL}}_r(\mathbb{F}_\mathfrak{l})/Z({\rm{SL}}_r(\mathbb{F}_\mathfrak{l}))$. On the other hand, we also have ${\rm{SL}}_r(A_\mathfrak{l})/N\cong {\rm{SL}}_r(\mathbb{F}_\mathfrak{l})/Z({\rm{SL}}_r(\mathbb{F}_\mathfrak{l}))$ Therefore, we get ${\rm{SL}}_r(A_\mathfrak{l})/N\cong {\rm{SL}}_r(A_\mathfrak{l})/N'$ and so $N'=N$. Otherwise, we may have $N \subsetneq NN' \triangleleft {\rm{SL}}_r(A_\mathfrak{l})$. Furthermore, we prove $NN'\neq {\rm SL}_r(A_\fl)$. Suppose $NN'={\rm SL}_r(A_\fl)$, then we have
$${\rm PSL}_r(\F_\fl)\cong {\rm SL}_r(A_\fl)/N'\cong NN'/N'\cong N/{N\cap N'}.$$
On the other hand, we look at the composition factors of ${\rm SL}_r(A_\fl)$. We have ${\rm{SL}}_r(A_\mathfrak{l})/N\cong {\rm{SL}}_r(\mathbb{F}_\mathfrak{l})/Z({\rm{SL}}_r(\mathbb{F}_\mathfrak{l}))\cong {\rm PSL}_r(\F_\fl)$, and the composition factor ${\rm PSL}_r(\F_\fl)$ appears only once. Therefore, the composition factors of $N$ are all abelian, either $\mathbb{Z}/p\mathbb{Z}$ or $\mathbb{Z}/r\mathbb{Z}$. It is impossible to have a quotient $N/N\cap N'$ of $N$ isomorphic to ${\rm PSL}_r(\F_\fl)$. Thus $NN'$ is a proper subgroup sits between ${\rm SL}_r(A_\fl)$ and $N$, which contradicts to the fact that ${\rm{SL}}_r(A_\mathfrak{l})/N\cong {\rm{SL}}_r(\mathbb{F}_\mathfrak{l})/Z({\rm{SL}}_r(\mathbb{F}_\mathfrak{l}))$ is simple.
\end{proof}

\begin{lem}\label{lem18}{\rm{(\cite{Zy11}, Lemma A.4)}}
Let $B_1$ and $B_2$ be finite groups and suppose that $H$ is a subgroup of $B_1\times B_2$ for which the two projections $p_1:H\rightarrow B_1$ and $p_2:H\rightarrow B_2$ are surjective. Let $N_1$ be the kernel of $p_2$ and $N_2$ be the kernel of $p_1$. We may view $N_1$ as a normal subgroup of $B_1$ and $N_2$ as a normal subgroup of $B_2$. Then the image of $H$ in $B_1/N_1\times B_2/N_2$ is the graph of the isomorphism $B_1/N_1\xrightarrow{\sim} B_2/N_2$.
\end{lem}

\begin{lem}\label{lem19}{\rm{(\cite{Zy11}, Lemma A.6)}}
Let $S_1,S_2,\cdots,S_k$ be finite groups with no non-trivial abelian quotients. Let $H$ be a subgroup of $S_1\times\cdots\times S_k$ such that each projection $H\rightarrow S_i\times S_j$ $(1\leqslant i< j \leqslant k)$ is surjective. Then $H=S_1\times\cdots\times S_k$.
\end{lem}

\begin{lem}\label{lem29}
Let $\mathfrak{l}_1$ and $\mathfrak{l}_2$ be two prime ideals of $A$, and set $\mathfrak{a}=\mathfrak{l}_1\mathfrak{l}_2$. Let $\phi$ be the Drinfeld module as before, and let $H=\bar{\rho}_{\phi,\mathfrak{a}}(G_F) \subseteq{\rm{GL}}_r(A/\mathfrak{a})$. Then $H$ satisfies the following properties:
\begin{enumerate}
\item $\det(H)=(A/\mathfrak{a})^*$;
\item the projections $p_1':H'\rightarrow {\rm{SL}}_r(A/\mathfrak{l}_1)$ and $p_2':H'\rightarrow {\rm{SL}}_r(A/\mathfrak{l}_2)$ are surjective, where $H'=H\cap {\rm{SL}}_r(A/\mathfrak{a})$;
\item the subring of $A/\mathfrak{a}$ generated by the set 
$$\mathcal{S}=\{{\rm{tr}}(h)^r/\det(h)\ |\ h\in H\}\cup\{\det(h)/{\rm{tr}}(h)^r\ |\ h\in H\ {\rm{with}}\ {\rm{tr}}(h)\in (A/\mathfrak{a})^*  \}$$ is exactly $A/\mathfrak{a}$.
\end{enumerate}
These three properties will imply $H={\rm{GL}}_r(A/\mathfrak{a})$.
\end{lem}
\begin{proof}
Condition $1$ follows from the fact that $\det\circ\bar{\rho}_{\phi,\fa}=\bar{\rho}_{C,\fa}$, the mod $\fa$ Galois representation of the Carlitz module. Condition $2$ follows from our result in section 5 and 6 that $\bar{\rho}_{\phi,\fl}$ is surjective for any prime ideal $\fl$. For condition $3$, we may take $c\in \mathbb{F}_q^*$ such that $\mathfrak{p}=T-c$ not equal to $\mathfrak{l}_1$ or $\mathfrak{l}_2$. By looking at the characteristic polynomial of $\bar{\rho}_{\phi,\fa}({\rm{Frob}}_\fp)$, we have 
$$\frac{\det(\bar{\rho}_{\phi,\mathfrak{a}}({\rm{Frob}}_{\mathfrak{p}}))}{{\rm{tr}}(\bar{\rho}_{\phi,\mathfrak{a}}({\rm{Frob}}_{\mathfrak{p}}))^r}\equiv \frac{T-c}{1^r}\equiv T-c \mod \mathfrak{a}.$$
The subring generated by these $T-c$ is equal to $A/\mathfrak{a}$.

The following shows how these three properties imply $H={\rm{GL}}_r(A/\mathfrak{a})$. Let $N_1'$ be the kernel of $p_2'$ and $N_2'$ be the kernel of $p_1'$. We may view $N_i'$ as a normal subgroup of ${\rm{SL}}_r(\mathbb{F}_{\mathfrak{l}_i})$. Lemma \ref{lem18} then implies the image of $H'$ in ${\rm{SL}}_r(\mathbb{F}_{\mathfrak{l}_1})/N_1'\times {\rm{SL}}_r(\mathbb{F}_{\mathfrak{l}_2})/N_2'$ is the graph of a group isomorphism ${\rm{SL}}_r(\mathbb{F}_{\mathfrak{l}_1})/N_1'\xrightarrow{\sim} {\rm{SL}}_r(\mathbb{F}_{\mathfrak{l}_2})/N_2'$. If one of $N_1'$ or $N_2'$ is the whole group, then the group isomorphism will force the other one to be the whole group. Therefore, if $N_1'={\rm{SL}}_r(\mathbb{F}_{\mathfrak{l}_1})$ or $N_2'={\rm{SL}}_r(\mathbb{F}_{\mathfrak{l}_2})$, we have $H'={\rm{SL}}_r(\mathbb{F}_{\mathfrak{l}_1})\times {\rm{SL}}_r(\mathbb{F}_{\mathfrak{l}_2})$. Combining with condition 1, we have $H={\rm{GL}}_r(A/\mathfrak{a})$.

Now we suppose both $N_i'$ are proper normal subgroups of ${\rm{SL}}_r(\mathbb{F}_{\mathfrak{l}_i})$. By the proof of Lemma \ref{lem28}, we can see that $N_i'\subseteq Z({\rm{SL}}_r(\mathbb{F}_{\mathfrak{l}_i}))$ for $i=1,2$. Since the order $|Z({\rm{SL}}_r(\mathbb{F}_{\mathfrak{l}_i}))|$ is either $1$ or $r$, the order $|N_i'|$ is either $1$ or $r$. Comparing the order on both sides of the isomorphism ${\rm{SL}}_r(\mathbb{F}_{\mathfrak{l}_1})/N_1'\xrightarrow{\sim} {\rm{SL}}_r(\mathbb{F}_{\mathfrak{l}_2})/N_2'$, we  can show that $|\mathbb{F}_{\mathfrak{l}_1}|=|\mathbb{F}_{\mathfrak{l}_2}|$, i.e. $\mathbb{F}_{\mathfrak{l}_1}$ and $\mathbb{F}_{\mathfrak{l}_2}$ are isomorphic fields.
For $i=1$ or $2$, define the projection $p_i: H\rightarrow {\rm{GL}}_r(\mathbb{F}_{\mathfrak{l}_i})$. Let $N_1$ be kernel of $p_2$ and $N_2$ be kernel of $p_1$, we may also view $N_i$ as normal subgroup of ${\rm{GL}}_r(\mathbb{F}_{\mathfrak{l}_i})$. Lemma \ref{lem18} then implies the image of $H$ into ${\rm{GL}}_r(\mathbb{F}_{\mathfrak{l}_1})/N_1\times {\rm{GL}}_r(\mathbb{F}_{\mathfrak{l}_2})/N_2$ is the graph of a group isomorphism ${\rm{GL}}_r(\mathbb{F}_{\mathfrak{l}_1})/N_1\xrightarrow{\sim} {\rm{GL}}_r(\mathbb{F}_{\mathfrak{l}_2})/N_2$.

As $N_i/N_i'\cong N_i{\rm{SL}}_r(\mathbb{F}_{\mathfrak{l}_i})/{\rm{SL}}_r(\mathbb{F}_{\mathfrak{l}_i})\subseteq {\rm{GL}}_r(\mathbb{F}_{\mathfrak{l}_i})/{\rm{SL}}_r(\mathbb{F}_{\mathfrak{l}_i})\cong \mathbb{F}_{\mathfrak{l}_i}^*$, we see that $N_i/N_i'$ and $N_i'$ are abelian. Thus $N_i$ is a solvable normal subgroup of ${\rm{GL}}_r(\mathbb{F}_{\mathfrak{l}_i})$. Furthermore, we see that $N_i\cong N_iZ({\rm{GL}}_r(\mathbb{F}_{\mathfrak{l}_i}))/Z({\rm{GL}}_r(\mathbb{F}_{\mathfrak{l}_i}))$ is a solvable normal subgroup of ${\rm{PGL}}_r(\mathbb{F}_{\mathfrak{l}_i})$. Hence we can deduce a normal series
$$\{1\}\lhd N_iZ({\rm{GL}}_r(\mathbb{F}_{\mathfrak{l}_i}))/Z({\rm{GL}}_r(\mathbb{F}_{\mathfrak{l}_i}))  \lhd {\rm{PGL}}_r(\mathbb{F}_{\mathfrak{l}_i}).$$
If ${\rm{PSL}}_r(\mathbb{F}_{\mathfrak{l}_i})={\rm{PGL}}_r(\mathbb{F}_{\mathfrak{l}_i})$, then we have
$N_i\subseteq Z({\rm{GL}}_r(\mathbb{F}_{\mathfrak{l}_i}))$ by the simplicity of ${\rm{PSL}}_r(\mathbb{F}_{\mathfrak{l}_i})$.
If ${\rm{PSL}}_r(\mathbb{F}_{\mathfrak{l}_i})\neq{\rm{PGL}}_r(\mathbb{F}_{\mathfrak{l}_i})$, then we consider the comsposition series 
$$\{1\}\lhd {\rm{PSL}}_r(\mathbb{F}_{\mathfrak{l}_i})  \lhd {\rm{PGL}}_r(\mathbb{F}_{\mathfrak{l}_i}).$$
Its factors are ${\rm{PSL}}_r(\mathbb{F}_{\mathfrak{l}_i})$ and $C_r$, the cyclic group of order $r$. Thus by the Jordan-H\"older theorem, $N_i\cong N_iZ({\rm{GL}}_r(\mathbb{F}_{\mathfrak{l}_i}))/Z({\rm{GL}}_r(\mathbb{F}_{\mathfrak{l}_i}))\cong C_r$. Let $\delta$ be a generator of $N_i\subset {\rm{GL}}_r(\mathbb{F}_{\mathfrak{l}_i})$. We know that $\delta$ satisfies the polynomial $x^r-1$. Since $q \equiv 1 \mod r$, $x^r-1$ splits over $\mathbb{F}_{{\mathfrak{l}_i}}$. The minimal polynomial of $\delta$ over $\mathbb{F}_{\mathfrak{l}_i}$ is a product of distinct degree one polynomials. Now we prove that the minimal polynomial of $\delta$ has to be a degree one polynomial.  Let $\zeta$ be a generator of $\F_q^*$ and $\omega=\zeta^{\frac{q-1}{r}}$. The roots of $x^r-1$ in $\F_q$ are $\omega^i$ with $0\leqslant i\leqslant r-1$. Suppose the minimal polynomial of $\delta\in N_i\subset {\rm{GL}}_r(\mathbb{F}_{\mathfrak{l}_i})$ is not a degree one polynomial, then $\delta$ can be written as a diagonal matrix with diagonal entries of the form $\omega^i$ but not all the same. Let's assume there are two distinct diagonal entries $\omega_1, \omega_2\in\{\omega^i, 0\leqslant i\leqslant r-1\}$ in $\delta$. The proof for general situation will follow the same process. Now we may write $\delta$ as
$$\delta=\left(\begin{array}{cccccc}\omega_1 &  &  &  &  &  \\ & \ddots &  &  &  &  \\ &  & \omega_1 &  &  &  \\ &  &  & \omega_2 &  &  \\ &  &  &  & \ddots &  \\ &  &  &  &  & \omega_2\end{array}\right).$$

Now since $N_i$ is a normal subgroup of ${\rm{GL}}_r(\mathbb{F}_{\mathfrak{l}_i})$, we can conjugate $\delta$ by a suitable permutation matrix to get a matrix in $N_i$ where two distinct diagonal entries interchanged while the other diagonal entries fixed. Namely, we have

$$\left(\begin{array}{cccccccc}\omega_1 &  &  &  &  &  &  &  \\ & \ddots &  &  &  &  &  &  \\ &  & \omega_1 &  &  &  &  &  \\ &  &  & \omega_2 &  &  &  &  \\ &  &  &  & \omega_1 &  &  &  \\ &  &  &  &  & \omega_2 &  &  \\ &  &  &  &  &  & \ddots &  \\ &  &  &  &  &  &  & \omega_2\end{array}\right)\in N_i$$

As $N_i$ is generated by $\delta$, this matrix is equal to $\delta^j$ for some $1\leqslant j\leqslant r-1$. By comparing the fixed entries in $\delta$ and $\delta^j$, we can deduce that $j$ must be equal to $1$ because $\omega^i$ has order equal to $r$ for $1\leqslant i\leqslant r-1$ (here we use $r$ is a prime number).  This implies $\omega_1=\omega_2$, a contradiction. Therefore, the minimal polynomial of $\delta$ is a degree one polynomial. Hence $\delta$ is a scalar matrix and we have $N_i\subseteq Z({\rm{GL}}_r(\mathbb{F}_{\mathfrak{l}_i}))$ in both cases.

 By taking further quotient, the image of $H$ into ${\rm{PGL}}_r(\mathbb{F}_{\mathfrak{l}_1})\times {\rm{PGL}}_r(\mathbb{F}_{\mathfrak{l}_2})$ is the graph of a group isomorphism 
$$\alpha: {\rm{PGL}}_r(\mathbb{F}_{\mathfrak{l}_1})\xrightarrow{\sim} {\rm{PGL}}_r(\mathbb{F}_{\mathfrak{l}_2}).$$

By \cite{Die80} Theorem 2, $\alpha$ can be lifted to an isomorphism
$$\tilde{\alpha}: {\rm{GL}}_r(\mathbb{F}_{\mathfrak{l}_1})\xrightarrow{\sim} {\rm{GL}}_r(\mathbb{F}_{\mathfrak{l}_2}).$$ 

Let $\sigma:\mathbb{F}_{\mathfrak{l}_1} \xrightarrow{\sim} \mathbb{F}_{\mathfrak{l}_2}$ be a field isomorphism and $\chi: {\rm{GL}}_r(\mathbb{F}_{\mathfrak{l}_1})\rightarrow \mathbb{F}_{\mathfrak{l}_2}^*$ be a homomorphism. Now we are able to create two group homomorphisms ${\rm{GL}}_r(\mathbb{F}_{\mathfrak{l}_1})\xrightarrow{\sim} {\rm{GL}}_r(\mathbb{F}_{\mathfrak{l}_2})$:

\begin{enumerate}
\item[(i)] $A\mapsto \chi(A)gA^{\sigma}g^{-1}$,
\item[(ii)] $A\mapsto \chi(A)g((A^{T})^{-1})^{\sigma}g^{-1}$,
\end{enumerate}

where $A\in{\rm{GL}}_r(\mathbb{F}_{\mathfrak{l}_1})$, $A^{\sigma}$ is the matrix that applies $\sigma$ to each entry of $A$, and $g\in {\rm{GL}}_r(\mathbb{F}_{\mathfrak{l}_2})$. By \cite{Die80} Theorem 1, there are $\sigma$, $\chi$, and $g$ such that $\tilde{\alpha}$ is  one of the homomorphisms above.
\ \\

\begin{lemma}: $\tilde{\alpha}$ must be of the first type. 
\end{lemma}
\begin{proof}
Suppose $\tilde{\alpha}$ is of second type, then we choose a degree $1$ prime ideal $\mathfrak{p}$ of $A$ different from $\mathfrak{l}_1$ and $\mathfrak{l}_2$. We consider the image of $\bar{\rho}_{\phi,\mathfrak{a}}({\rm{Frob}}_{\mathfrak{p}})\in H$ in ${\rm{PGL}}_r(\mathbb{F}_{\mathfrak{l}_1})\times {\rm{PGL}}_r(\mathbb{F}_{\mathfrak{l}_2})$ under
$$
\begin{array}{ccc}
\ \ \ \ \ \ \  H &\hookrightarrow &{\rm{PGL}}_r(\mathbb{F}_{\mathfrak{l}_1})\times {\rm{PGL}}_r(\mathbb{F}_{\mathfrak{l}_2})\\
\ \ \ \ \ \ \  \bar{\rho}_{\phi,\mathfrak{a}}({\rm{Frob}}_{\mathfrak{p}}) &\mapsto & (\bar{\rho}_{\phi,{\mathfrak{l}}_1}({\rm{Frob}}_{\mathfrak{p}})\cdot Z({\rm{GL}}_r(\mathbb{F}_{\mathfrak{l}_1})),\bar{\rho}_{\phi,{\mathfrak{l}_2}}({\rm{Frob}}_{\mathfrak{p}})\cdot Z({\rm{GL}}_r(\mathbb{F}_{\mathfrak{l}_2})))\\
&&\| \\
&&(\bar{\rho}_{\phi,{\mathfrak{l}}_1}({\rm{Frob}}_{\mathfrak{p}})\cdot Z({\rm{GL}}_r(\mathbb{F}_{\mathfrak{l}_1})),\tilde{\alpha}(\bar{\rho}_{\phi,{\mathfrak{l}}_1}({\rm{Frob}}_{\mathfrak{p}}))\cdot Z({\rm{GL}}_r(\mathbb{F}_{\mathfrak{l}_2})))\\
&&\| \\
&&(\bar{\rho}_{\phi,{\mathfrak{l}}_1}({\rm{Frob}}_{\mathfrak{p}})\cdot Z({\rm{GL}}_r(\mathbb{F}_{\mathfrak{l}_1})),g(((\bar{\rho}_{\phi,{\mathfrak{l}}_1}({\rm{Frob}}_{\mathfrak{p}}))^{\sigma})^{T})^{-1}g^{-1}\cdot Z({\rm{GL}}_r(\mathbb{F}_{\mathfrak{l}_2})))
\end{array}
$$
Therefore, $\bar{\rho}_{\phi,{\mathfrak{l}_2}}({\rm{Frob}}_{\mathfrak{p}})\cdot Z({\rm{GL}}_r(\mathbb{F}_{\mathfrak{l}_2}))$ and $g(((\bar{\rho}_{\phi,{\mathfrak{l}}_1}({\rm{Frob}}_{\mathfrak{p}}))^{\sigma})^{T})^{-1}g^{-1}\cdot Z({\rm{GL}}_r(\mathbb{F}_{\mathfrak{l}_2}))$ are the same coset in ${\rm{PGL}}_r(\mathbb{F}_{\mathfrak{l}_2})$. For each element in these two cosets, we can compute its trace. All elements in $\bar{\rho}_{\phi,{\mathfrak{l}_2}}({\rm{Frob}}_{\mathfrak{p}})\cdot Z({\rm{GL}}_r(\mathbb{F}_{\mathfrak{l}_2}))$ has nonzero trace because the characteristic polynomial of $\bar{\rho}_{\phi,{\mathfrak{l}_2}}({\rm{Frob}}_{\mathfrak{p}})$ is $x^r+x^{r-1}-\mathfrak{p}$ (mod $\mathfrak{l}_2$). However, the trace of all elements in $g(((\bar{\rho}_{\phi,{\mathfrak{l}}_1}({\rm{Frob}}_{\mathfrak{p}}))^{\sigma})^{T})^{-1}g^{-1}\cdot Z({\rm{GL}}_r(\mathbb{F}_{\mathfrak{l}_2}))$ are equal to zero because the characteristic polynomial of $(\bar{\rho}_{\phi,{\mathfrak{l}}_1}({\rm{Frob}}_{\mathfrak{p}}))^{-1}$ is $x^r-\frac{1}{\mathfrak{p}}x-\frac{1}{\mathfrak{p}}$ (mod $\mathfrak{l}_1$). Hence we have a contradiction and $\tilde{\alpha}$ must be of first type.
\end{proof}

Writing $\tilde{\alpha}(A)=\chi(A)gA^{\sigma}g^{-1}$  for all $A\in {\rm{GL}}_r(\mathbb{F}_{\mathfrak{l}_1})$, we have
$$\frac{{\rm{tr}}(\tilde{\alpha}(A))^r}{\det(\tilde{\alpha}(A))}=\sigma\left( \frac{{\rm{tr}}(A)^r}{\det(A)}            \right).$$
Therefore, for each element $(h_1,h_2)\in H$, we have
$$\frac{{\rm{tr}}(h_2)^r}{\det(h_2)}=\sigma\left( \frac{{\rm{tr}}(h_1)^r}{\det(h_1)}            \right).$$

Let $W=\{(x_1,x_2)\ |\ \sigma(x_1)=x_2\}$ be the subring of $A/\mathfrak{a}\cong\mathbb{F}_{\mathfrak{l}_1}\times \mathbb{F}_{\mathfrak{l}_2}$. We have $\mathcal{S}\subseteq W$. However, $W\neq A/\mathfrak{a}$ by counting cardinality on both sides. Thus we get a contradiction from the assumption that $N_i'$ is a proper normal subgroup of ${\rm{SL}}_r(\mathbb{F}_{\mathfrak{l}_i})$ for $i=1$ or $2$. The proof is complete.

\end{proof}

\begin{lem}\label{lem30}
Let $\mathfrak{l}_1$ and $\mathfrak{l}_2$ be two finite places of $F$. Define
$$\rho:G_F\rightarrow {\rm{GL}}_r(A_{\mathfrak{l}_1})\times {\rm{GL}}_r(A_{\mathfrak{l}_2}),\ \ \sigma\mapsto (\rho_{\phi,\mathfrak{l}_1}(\sigma),\rho_{\phi,\mathfrak{l}_2}(\sigma)).$$
Then $\rho(G_F)={\rm{GL}}_r(A_{\mathfrak{l}_1})\times {\rm{GL}}_r(A_{\mathfrak{l}_2})$.
\end{lem}

\begin{proof}
It's enough to show that for any positive integers $n_1$ and $n_2$, we have
$$\bar{\rho}_{\phi,\mathfrak{a}}(G_{F^{\rm{ab}}})={\rm{SL}}_r(A/\mathfrak{a})$$
where $\mathfrak{a}=\mathfrak{l}_1^{n_1}\mathfrak{l}_2^{n_2}$. 

Suppose the equality doesn't hold, then we can apply Lemma \ref{lem18} with $H=\bar{\rho}_{\phi,\mathfrak{a}}(G_{F^{\rm{ab}}})$, $B_1={\rm{SL}}_r(A/\mathfrak{l}_1^{n_1})$ and $B_2={\rm{SL}}_r(A/\mathfrak{l}_2^{n_2})$.  The image of $H$ in ${\rm{SL}}_r(A/\mathfrak{l}_1^{n_1})/N_1\times {\rm{SL}}_r(A/\mathfrak{l}_2^{n_2})/N_2$ is the graph of the isomorphism ${\rm{SL}}_r(A/\mathfrak{l}_1^{n_1})/N_1\xrightarrow{\sim} {\rm{SL}}_r(A/\mathfrak{l}_2^{n_2})/N_2$. From the second paragraph in the proof of Lemma \ref{lem29}, we know that ${\rm{SL}}_3(A/\mathfrak{l}_i^{n_i})/N_i$ are not trivial for $i=1, 2$. By Lemma \ref{lem28}, $N_i$ is a subgroup of 
$$\bar{N}:=\left\{B\in{\rm{SL}}_r(A/\mathfrak{l}_i^{n_i})| B\equiv\delta \cdot I_r \mod\mathfrak{l}, \ {\rm{where}}\ \delta\in \mathbb{F}_{\mathfrak{l}}\ {\rm{satisfies}}\ \delta^r=1              \right\}.$$ 
Taking further quotient, the image of $H$ in ${\rm{SL}}_r(A/\mathfrak{l}_1)/Z({\rm{SL}}_r(A/\mathfrak{l}_1))\times {\rm{SL}}_r(A/\mathfrak{l}_2)/Z({\rm{SL}}_r(A/\mathfrak{l}_2))$ is the graph of an isomorphism ${\rm{SL}}_r(A/\mathfrak{l}_1)/Z({\rm{SL}}_r(A/\mathfrak{l}_1))\xrightarrow{\sim} {\rm{SL}}_r(A/\mathfrak{l}_2)/Z({\rm{SL}}_r(A/\mathfrak{l}_2))$. However, Lemma \ref{lem29} shows the reduction of $H$ modulo $\mathfrak{l}_1\mathfrak{l}_2$ is the whole group ${\rm{SL}}_r(A/\mathfrak{l}_1)\times {\rm{SL}}_r(A/\mathfrak{l}_2)$. Thus the image of $H$ into ${\rm{SL}}_r(A/\mathfrak{l}_1)/Z({\rm{SL}}_r(A/\mathfrak{l}_1))\times {\rm{SL}}_r(A/\mathfrak{l}_2)/Z({\rm{SL}}_r(A/\mathfrak{l}_2))$ should be the whole group, this gives a contradiction.
\end{proof}

\subsection{Proof of the main theorem}
\begin{thmm}
Let $q=p^e$ be a prime power, $A=\mathbb{F}_q[T]$, and $F=\mathbb{F}_q(T)$. Assume $r\geqslant 3$ is a prime number and $q\equiv 1 \mod r$, there is a constant $c=c(r)\in \mathbb{N}$ depending only on $r$ such that for $p>c(r)$ the following statement is true ${\rm{:}}$

Let $\phi$ be a Drinfeld $A$-module over $F$ of rank $r$ with generic characteristic, which  is defined by $\phi_T=T+\tau^{r-1}+T^{q-1}\tau^r$. Then the adelic Galois representation 
$${\rho}_{\phi}:{\rm{Gal}}(\mathbb{F}_q(T)^{{\rm{sep}}}/\mathbb{F}_q(T))\longrightarrow \varprojlim_{\mathfrak{a}}{\rm{Aut}}(\phi[\mathfrak{a}])\cong {\rm{GL_r}}(\widehat{A})$$
 is surjective.
\end{thmm}
\begin{proof}

Firstly, $\det\circ\rho_{\phi}=\rho_{C}$ is the adelic representation of Carlitz module. We know that the adelic representation of the Carlitz module is surjective, so $\det\circ\rho_{\phi}=\widehat{A}^*$. It suffices for us to prove $\rho_{\phi}(G_{F^{\rm{ab}}})={\rm{SL}}_r(\widehat{A})$. It's equivalent to show that for every nonzero ideal $\mathfrak{a}=\mathfrak{l}_1^{n_1}\mathfrak{l}_2^{n_2}\cdots\mathfrak{l}_k^{n_k}$ of $A$ , we have $$\bar{\rho}_{\phi,\mathfrak{a}}(G_{F^{\rm{ab}}})={\rm{SL}}_r(A/\mathfrak{a})\cong \prod_{i}{\rm{SL}}_r(A/\mathfrak{l}_i^{n_i}).$$
By Lemma \ref{lem28}, we know that each ${\rm{SL}}_r(A/\mathfrak{l}_i^{n_i})$ has no nontrivial abelian quotient. Thus we can apply Lemma \ref{lem19} once we can prove that each projection
$\bar{\rho}_{\phi,\mathfrak{a}}(G_{F^{\rm{ab}}})\rightarrow {\rm{SL}}_r(A/\mathfrak{l}_i^{n_i})\times {\rm{SL}}_r(A/\mathfrak{l}_j^{n_j})$ is surjective for $1\leqslant i< j\leqslant k$. The surjectivity of each projection is proved by Lemma \ref{lem30}. Thus Lemma \ref{lem19} implies $\bar{\rho}_{\phi,\mathfrak{a}}(G_{F^{\rm{ab}}})={\rm{SL}}_r(A/\mathfrak{a})$ for every nonzero ideal $\mathfrak{a}$ of $A$, the proof of theorem is complete.

\end{proof}

\section*{Acknowledgment}

The author would like to thank his thesis advisor, Professor Mihran Papikian, for helpful discussions and suggestions on earlier drafts of this paper.

\begin{appendices}
In the Appendices, we restate Aschbacher's theorem in \cite{BHR13} and focus on the classification of subgroups in general linear group over a finite field.

\section{Notations and Aschbacher classes}
\subsection{Notation}
At here, $A,B,G$ are groups, and $a,b,n$ belong to $\mathbb{N}$. 
\begin{itemize}

\item ${\rm{GL}}_n(q)={\rm{GL}}_n(\mathbb{F}_q)$
\item $Z(G)$ denote the center of $G$.
\item $[G,G]$ or $G'$ denote the derived subgroup of $G$.
\item For $n>1$, $G^{(n)}=[G^{(n-1)},G^{(n-1)}]$. $G^{\infty}=\bigcap_{i\geqslant 0}G^{(i)}$.

\item$A.B=$ an extension of $A$ by $B$ with unspecified splittness.
\item$A\wr B=$ wreath product of $A$ by a permutation group $B$.

\item ${\rm{E}}_{p^n}$ or just $p^n$= elementary abelian group of order $p^n$.

\item For an elementary abelian group $A$, $A^{m+n}=$ a group with elementary abelian normal subgroup $A^m$ such that the quotient is isomorphic to $A^n$.
\end{itemize}

\subsection{Aschbacher classes}

Let $H$ be a subgroup of ${\rm{GL}}_n(q)$, where $n\geqslant 3$. In this subsection we give a summary of information when $H$ lies in an Aschbacher class. For a complete edition of Aschbacher classes and the definition of each class we refer to chapter 2 in \cite{BHR13}.

\begin{itemize}
\item $H$ lies in class $\mathcal{C}_1\Rightarrow$ $H$ stabilizes a proper non-zero subspace of $\mathbb{F}_q^n$.

\item $H$ lies in class $\mathcal{C}_2\Rightarrow$  there is a direct sum decomposition $\mathcal{D}$ of $\mathbb{F}_q^n$ into $t$ subspaces, each of dimension $m=n/t$: 
$$\mathcal{D}: \mathbb{F}_q^n=V_1\oplus V_2\oplus \cdots \oplus V_t,\ \ {\rm{where}}\ t\geqslant 2$$
The action of $H$ on $\mathbb{F}_q^n$ is of the type ${\rm{GL}}_m(q)\wr S_t$.

\item $H$ lies in class $\mathcal{C}_3\Rightarrow$ There is a prime divisor $s\geqslant 2$ of $n$ and $m=n/s$ such that $\mathbb{F}_q^n$ has a $\mathbb{F}_{q^s}$-vector space structure and $H$ acts $\mathbb{F}_{q^s}$-semilinear on $\mathbb{F}_q^n$. The action of $H$ on $\mathbb{F}_q^n$ is of the type ${\rm{GL}}_m(q^s)$

\item $H$ lies in class $\mathcal{C}_4\Rightarrow$ $H$ preserves a tensor product decomposition $\mathbb{F}_q^n=V_1\otimes V_2$, where $V_1$(resp. $V_2$) is a $\mathbb{F}_q$-subspace of $\mathbb{F}_q^n$ of dimension $n_1$(resp. $n_2$) and $1<n_1<\sqrt{n}$. The action of $H$ on $\mathbb{F}_q^n$ is of the type ${\rm{GL}}_{n_1}(q)\otimes {\rm{GL}}_{n_2}(q)$.

\item $H$ lies in class $\mathcal{C}_5\Rightarrow$ $H$ acts on $\mathbb{F}_q^n$ absolutely irreducible and there is a subfield $\mathbb{F}_{q_0}$ of $\mathbb{F}_q$ such that a conjugate of $H$ in ${\rm{GL}}_n(q)$ is a subgroup of $<Z({\rm{GL}}_n(q)), {\rm{GL}}_{n}(q_0)>$.

\item $H$ lies in class $\mathcal{C}_6\Rightarrow$ $n=r^m$ with $r$ prime. There is an absolutely irreducible group $E$ such that $E\trianglelefteq H \leqslant N_{ {\rm{GL}}_{n}(q)}(E)$. Here $E$ is either an extraspecial $r$-group or a $2$-group of symplectic type.
The action of $H$ on $\mathbb{F}_q^n$ is of the type ${r^{1+2m}}. {\rm{Sp}}_{2m}(r)$ when $n$ is odd. And $H$ is of the type ${2^{2+2m}}. {\rm{Sp}}_{2m}(2)$ when $n$ is even.

\item $H$ lies in class $\mathcal{C}_7\Rightarrow$ $H$ preserves a tensor induced decomposition $\mathbb{F}_q^n=V_1\otimes V_2\otimes\cdots\otimes V_t$ with $t\geqslant 2$, dim$V_i=m$ and $n=m^t$. The action of $H$ on $\mathbb{F}_q^n$ is of the type ${\rm{GL}}_m(q)\wr S_t$. Here the wreath product is a tensor wreath product, which is a quotient of standard wreath product.

\item $H$ lies in class $\mathcal{C}_8\Rightarrow$ $H$ preserves a non-degenerate classical form on $\mathbb{F}_q^n$ up to scalar multiplication. By classical form we mean symplectic form, unitary form or quadratic form.

\item $H$ lies in class $\mathcal{S} \Rightarrow$ $H$ doesn't contain ${\rm{SL}}_n(q)$ and $H^{\infty}$ acts on $\mathbb{F}_q^n$ absolutely irreducibly.

\end{itemize}

\section{Note on Aschbacher's theorem}

\begin{thm}{\rm{(Special case of Aschbacher's Theorem)}}\label{asch}

Let $H$ be a subgroup of $ {\rm{GL}}_{n}(q)$, then $H$ lies in one of the Aschbacher classes $\mathcal{C}_i$ or $\mathcal{S}$.

\end{thm}

In \cite{BHR13} Theorem 2.2.19, the authors state a detailed version of Aschbacher's Theorem and describe the structure of $H\cap  {\rm{SL}}_{n}(q)$ when $H$ maximal in ${\rm{GL}}_{n}(q)$ that lies in a class $\mathcal{C}_i$ for $1\leqslant i \leqslant 8$.

\end{appendices}

\section*{Conflict of interest statement}
On behalf of the author, the author states that there is no conflict of interest.

\bibliographystyle{alpha}
\bibliography{DM_of_prime_rank_rev.3}

\end{document}